%% file: Legendre_xxx.tex
\documentclass[a4paper,10pt]{article}

\input{SWi}

\title{The Legendre-Fenchel transform from a category theoretic perspective}
\author{Simon Willerton}
\date{}

\begin{document}
\maketitle

\begin{abstract}
The Legendre-Fenchel transform is a classical piece of mathematics with many applications.  In this paper we show how it arises in the context of category theory using categories enriched over the extended real numbers $\overline{ \mathbb{R}}:=[-\infty,+\infty]$.  A key ingredient is Pavlovic's ``nucleus of a profunctor'' construction.  The pairing between a vector space and its dual can be viewed as an $\overline {\mathbb{R}}$-profunctor; the construction of the nucleus of this profunctor is the construction of a lot of the theory of the Legendre-Fenchel transform.  For a relation between sets viewed as a $\{\true,\false\}$-valued profunctor, the construction of the nucleus is the construction of the Galois connection associated to the relation.

One insight given by this approach is that the relevant structure on the function spaces involved in the Legendre-Fenchel transform is something like a metric but is asymmetric and can take negative values.  This `$\overline {\mathbb{R}}$-structure' is a considerable refinement of the usual partial order on real-valued function space and it allows a natural interpretation of Toland-Singer duality and of the two tropical module structures on the set of convex functions.
\end{abstract}

\tableofcontents
\section{Introduction}
One role of category theory is as a kind of metamathematics: it gives a language for describing how seemingly different constructions in disparate parts of mathematics are really the same.  Until the 1970s, areas such as logic, algebra, algebraic geometry and algebraic topology were the areas in which category theory had most reach, but then Lawvere~\cite{Lawvere:MetricSpaces} showed using \emph{enriched} category theory  that category theory was entwined with metric space theory.  

Following Lawvere~\cite{Lawvere:TakingCategoriesSeriously}
and taking {enriched} categories seriously, we will see here how, from the pairing between a vector space and its dual, the \emph{nucleus of a profunctor construction}~\cite{Pavlovic:FCA} gives rise to a large amount of the theory around the Legendre-Fenchel transform and convex functions.  We will also see how the  Galois correspondence arising from a relation between sets is an example of this nucleus from a profunctor construction; other examples of this construction include the Isbell completion of a category, the Dedekind-MacNeille completion for posets, the fuzzy concept lattice arising from a fuzzy context~\cite{Belohlavek:FuzzyGaloisConnections,Pavlovic:FCA}, the tropical span of a tropical matrix~\cite{DevelinSturmfels:Tropical} and the directed tight span for generalized metric spaces~\cite{Willerton:Isbell}.

\subsection{The basic idea}
Here we will recall the construction of Galois connection from a relation  --- which is a fundamental construction across mathematics --- and also recall the construction of the Legendre-Fenchel transform and then briefly explain how they have a common refined generalization in the nucleus of a profunctor construction.

\subsubsection{Relations and Galois connections.}
Suppose you have two sets $G$ and $M$ and a relation $\I $ between them.   Two examples to bear in mind here are the following:
\begin{enumerate}
    \item from algebraic geometry, $G=\CC^n$, $M=\CC[x_1,\dots x_n]$ and $x\Irel p$ if $p(x)=0$;
    \item from number theory, $G$ is a field $L$ equipped with a subfield $K\subset L$, $M=\Aut(L,K)$ the field automorphisms fixing $K$, and $\ell\Irel \phi$ if $\phi(\ell)=\ell$.
\end{enumerate}
We write  $\PP(G)$ for the poset of subsets of $G$ ordered by subset inclusion.  Similarly, we write $\PP(M)^\op$ for the poset of subsets of $M$ with the opposite order.  From $\I $ we get a pair of order preserving functions
  \[\I ^\ast\colon \PP(G)\rightleftarrows \PP(M)^\op\colon \I _\ast,\]
where 
\begin{align*}
\I^\ast(S)&\coloneqq \{m\in M\mid s \Irel  m\ \text{for all } s\in S\},\\
\I_\ast(T)&\coloneqq \{g\in G\mid g\Irel t\ \text{for all } t\in T\}.
\end{align*}
These form a \define{Galois connection} in that 
\[ S \subseteq \I _\ast(T) \iff  \I ^\ast(S)\supseteq T,\]
as both sides are saying that every element of $T$ is related to every element of $S$.
We then get a \define{closure operation} on each powerset:
  \[\I _\ast \I ^\ast\colon \PP(G)\to \PP(G),\quad \I ^\ast \I _\ast\colon \PP(M)^{\op}\to \PP(M)^{\op}.\]
These are idempotent --- $(\I _\ast \I ^\ast)^2=\I _\ast \I ^\ast$ and $(\I ^\ast \I _\ast)^2=\I ^\ast \I _\ast$ --- and they satisfy $S\subset \I _\ast \I ^\ast(S)$ and $T\supset \I ^{\ast} \I_{\ast}(T)$ for all $S\subset G$, $T\subset M$, thus they merit the name closure operations.  The \define{closed subsets} are those which are fixed by the closure operations, i.e.~those that are equal to their closure.  The totality of closed subsets gives two posets  $\PP_{\cl}(G)$ and   $\PP_{\cl}(G)$.
The Galois connection restricts to a \define{Galois correspondence}, i.e.~a duality, between the closed subsets of $G$ and the closed subsets of $M$:
\[\I ^\ast\colon \PP_{\cl}(G)\cong \PP_{\cl}(M)^\op\colon \I _\ast.\]

We can look at what this gives in the examples mentioned above.
\begin{enumerate}
\item In the algebraic geometry example,
the closed subsets of $\CC^n$ are precisely the algebraic sets (that is those defined by a set of polynomials) and the closed subsets of $\CC[x_1,\dots,x_n]$ are precisely the radical ideals.  The Galois correspondence is the fundamental correspondence of algebraic geometry.
\item In the number theory example, provided the extension is Galois, the closed subsets of $L$ are the intermediate fields between $L$ and $K$, while the closed subsets of $\Aut(L,K)$ are precisely the Krull-closed subgroups.  The Galois correspondence is \emph{the} classical Galois correspondence.
\end{enumerate}
\subsubsection{The Legendre-Fenchel transform.}
Suppose that $V$ is a real vector space, $\dual{V}$ is its linear dual and $\Fun(V,\Rbar)$ denotes the space of functions from $V$ to the extended real numbers $\Rbar:=[-\infty,+\infty]$.  There is a standard  pair of transforms between function spaces
\[\fenchb \colon \Fun({V},\Rbar)\rightleftarrows \Fun(\dual{V},\Rbar)\colon \rfenchb,\]
given by
  \[\fenchbop{f}(k)\coloneqq \sup_{x\in V}\big\{\langle k,x\rangle -f(x)\big\},\quad 
  \rfenchbop{g}(x)\coloneqq \sup_{k\in \dual{V}}\big\{\langle k,x\rangle -g(k)\big\}.\]
The transform $\fenchb$ is the Legendre-Fenchel transform and we will refer to $\rfenchb$ as the reverse Legendre-Fenchel transform.%
\footnote{In Section~\ref{sec:LF} below, I will revert to the standard notation of writing $\fenchbop{f}$ as $\fench{f}$ and similarly $\rfenchbop{g}$ as $\fench{g}$.}
The function spaces each have a partial order on them using the total order on $\Rbar$, for $f_{1},f_{2}\colon V\to \Rbar$,
  \[f_1\greeq f_2 \quad\iff \quad f_1(x)\ge f_2(x) \quad\text{for all }x\in V.\]
We use the opposite order for the functions on $\dual{V}$.  With respect to these orders the two transforms form a Galois connection, so we get a closure operation on each function space.
  \[\rfenchb\fenchb\colon \Fun({V},\Rbar) \to \Fun({V},\Rbar); \quad 
    \fenchb\rfenchb\colon \Fun(\dual{V},\Rbar) \to \Fun(\dual{V},\Rbar).\]
These are idempotent and, 
for instance, $\rfenchb\fenchbop{f}\greeq f$.  In fact, $\rfenchb\fenchbop{f}$ is the pointwise-smallest, lower semi-continuous, convex function which dominates $f$; the lower-semicontinous part is only relevant when the function takes infinite values.  Now we can look at the functions fixed by these closure operators and these are precisely the lower-semicontinuous convex functions.  From this it follows that the Lengendre-Fenchel transform and its reverse restrict to a bijection, actually a Galois correspondence: using $\Cvx$ to denote the lower semi-continuous, convex functions we have
  \[\fenchb\colon\Cvx(V,\Rbar)\cong \Cvx(\dual{V},\Rbar)\,\colon\!\rfenchb.\]
This is known as Legendre-Fenchel duality.  It crops up in many areas of mathematics for instance in mathematical physics to translate between Hamiltonian and Lagrangian formalisms (see, e.g.,~\cite{Arnold:MathematicalMethods}), in optimization theory (see, e.g.,~\cite{Rockafellar}) and in large deviation theory (see, e.g.,~\cite{Tourchette}).

\subsubsection{The general construction}
The two constructions presented are clearly similar, they both involve a Galois connection to construct a duality (or a Galois correspondence if you prefer).  However, the similarity goes much deeper and it can be illuminated in the light of enriched category theory, as both of these constructions are examples of what, following Pavlovic~\cite{Pavlovic:FCA}, we will call the nucleus of a profunctor construction; Shen and Zhang~\cite{ShenZhang:CategoriesOverAQuantaloid} also considered the construction.   The terms used in the following paragraph will be explained in Section~\ref{sec:posetenriched}.

  For the construction you start with a profunctor between two sets or enriched categories, in the first case this is the relation between $G$ and $M$, thought of a function $G\times M\to \{\true,\false\}$, and in the second case this is the tautological pairing $V\times \dual{V}\to \R\subset \Rbar$ between the vector space and its dual.   You then consider the categories of presheaves on your original categories, in the first case this leads to the posets of subsets of $G$ and $M$ and in the second case to the spaces of functions on $V$ and $\dual{V}$.  From the profunctor you construct an enriched adjunction between the presheaf categories.  In the one case you get the Galois connection between the sets of subsets, in the other case you get the Legendre-Fenchel transform and its reverse.  The nucleus is then extracted as the parts of the presheaf categories on which the adjunction restricts to an equivalence.  In the one case you get the Galois correspondence, and often an interesting duality, in the other case you get the Legendre-Fenchel duality between convex functions.

Here the set of classical truth values $\{\true,\false\}$ and the set of extended real numbers $\Rbar$ both form examples of what a category theorist might call a complete and cocomplete, closed monoidal poset; other folk might call such a thing a commutative quantale or a complete, commutative, idempotent semiring.  In this paper we present the theory of categories and profunctors enriched over such things paying particular attention to these two main examples.

\subsection{What we gain}
A fair question to ask at this point is ``What do we gain from viewing Legendre-Fenchel duality as a profunctor nucleus?''  It could be a case of translating something familiar to an unfamiliar language for no real gain.  In fact, we gain a few things.  On the one hand we see that from a category theory point of view, the Legendre-Fenchel transform is inevitable once you have the pairing between a vector space and its dual.  You can see how much of the theory is just formal nonsense and how much is specific to the case in hand.  As the profunctor nucleus arises in various places, such as in tight spans for metric spaces, there is the potential for building bridges between different disciplines.

On the other hand, on a more concrete level, a fundamental insight that we gain is that the set of functions $\Fun(V,\Rbar)$ should be considered not as a poset but as an $\Rbar$-space, something like an asymmetric metric space with possibly negative distances.  There is an $\Rbar$-metric, like a very generalized notion of distance, 
 $\dd\colon \Fun(V,\Rbar)\times \Fun(V,\Rbar)\to \Rbar$ given by
 \[\dd(f_1,f_2)\coloneqq \sup_{x\in V} \{f_2(x)-f_1(x)\},\]
where some care is needed to understand the difference between infinite quantities, for example $(-\infty) - (-\infty)=-\infty$ and  $(+\infty) - (+\infty)=-\infty$, see Section~\ref{sec:tables} for details.
This satisfies the triangle inequality and $\dd(f,f)=0$ if $f$ takes a finite value somewhere, otherwise $\dd(f,f)=-\infty$.  However, it is not symmetric.

Underlying any such $\Rbar$-metric is a preorder $\greeq$, defined as follows:
  \[f_1\greeq f_2 \quad\iff \quad 0\ge \dd(f_1,f_2).\]
Unpacking the definition this gives the usual partial order on the space of functions described above.
So this $\Rbar$-metric on $\Fun(V,\Rbar)$ is a refinement of the usual partial order.

With respect to this $\Rbar$-metric, we find (Theorem~\ref{thm:LegendreIsShort}) that the Legendre-Fenchel transform is a distance non-increasing function so that $ \dd(f_1,f_2)\ge\dd(\fenchbop{f_1},\fenchbop{f_2})$, in other words,
 \[\sup_{x\in V}\{f_2(x)-f_1(x)\} \ge \sup_{k\in \dual{V}}\{ \fenchbop{f_1}(k) - \fenchbop{f_2}(k)\}.\]
We also learn (Theorem~\ref{thm:LFadjunction}) that the Legendre-Fenchel transform and its reverse form an adjunction which means that $\dd(\fenchbop{f},g)= \dd(f,\rfenchbop{g})$, in other words,
\[\sup_{k\in \dual{V}}\{ \fenchbop{f}(k) - g(k)\}= \sup_{x\in V}\{\rfenchbop{g}(x)-f(x)\}.\]
This is a refinement of the statement that the Legendre-Fenchel transform and its reverse form a Galois connection, i.e.~that $f\greeq \rfenchbop{g}$ if and only if $g\greeq\fenchbop{f}$. 

In the case that it is restricted to convex functions then we find (Theorem~\ref{thm:TolandSingerWeak}) that the Legendre-Fenchel transform is actually a distance preserving map, meaning $\dd(\fenchbop{f_1},\fenchbop{f_2})=\dd(f_1,f_2)$, in other words,
\[\sup_{k\in \dual{V}}\{ \fenchbop{f_1}(k) - \fenchbop{f_2}(k)\}= \sup_{x\in V}\{f_2(x)-f_1(x)\}.\]
This is not part of the standard treatment of the Legendre-Fenchel transform, but it is known as Toland-Singer duality.  It is a refinement of the standard order theoretic statement that for convex functions $f_1$ and $f_2$ we have $f_1\greeq f_2 \iff \fenchbop{f_2}\greeq \fenchbop{f_1}$.  The fact that it is stronger than that is illustrated in Figure~\ref{Figure:TolandSinger}.
\begin{figure}
\centering
\includegraphics{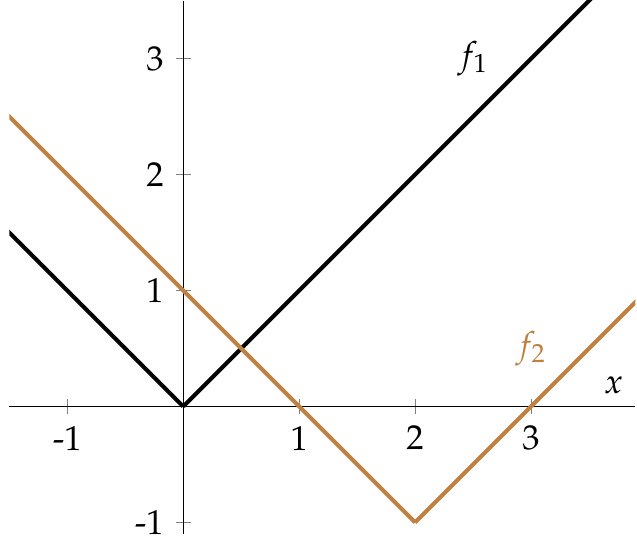}\qquad\includegraphics{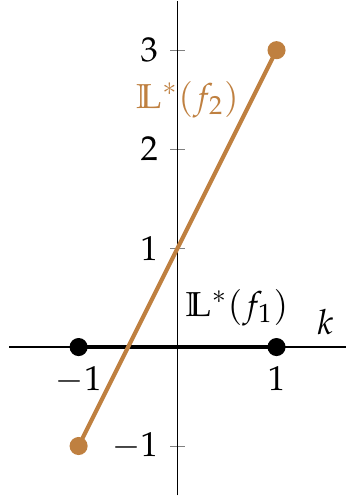}
{}
\caption{Illustrating Toland-Singer duality.  Note $\fenchbop{f_1}(k)=\fenchbop{f_2}(k)=+\infty$ when $\lvert k\rvert >1$.}
\label{Figure:TolandSinger}
\end{figure}

In the first picture of Figure~\ref{Figure:TolandSinger} we see two convex functions $f_1,f_2\colon \R\to \Rbar$.  Neither function dominates the other so the fact that the Legendre-Fenchel transform is order reversing tells us nothing about the transforms of these functions.  We see, however, that $\dd(f_1,f_2)=1$ and $\dd(f_2,f_1)=3$, so Toland-Singer duality tells us that $\dd(\fenchbop{f_1},\fenchbop{f_2})=1$ and $\dd(\fenchbop{f_2},\fenchbop{f_1})=3$ which we clearly see in the second picture.  The distance on the first picture is the maximum climb between graphs and the distance on the second picture is the maximum fall.

The nucleus of a profunctor is complete and cocomplete in a categorical sense.  What that means here in particular is that the space of convex functions has products, coproducts, tensors and cotensors (Theorem~\ref{thm:CvxLimitsColimits}).

The product of two convex functions is just  their pointwise maximum:
  \[f_1\sqcap f_2 (x)\coloneqq\max(f_1(x),f_2(x)).\]
The coproduct is the lower semi-continuous, convex hull of their pointwise minimum:
  \[f_1\sqcup f_2 \coloneqq\rfenchb\fenchbop{\min(f_1,f_2)}.\]
In the order theoretic world these are the least upper bound and greatest lower bounds of the functions, but in this metric world they satisfy stronger properties, namely
  \begin{align*}
  \dd(f,f_1\sqcap f_2)&=\max(\dd(f,f_1),\dd(f,f_2))\\
  \dd(f_1\sqcup f_2,f)&=\max(\dd(f_1,f),\dd(f_2,f))
  \end{align*}

The tensor product $\odot$ and cotensor product $\cotensor$ are actions of the monoid $(\Rbar,+)$ on the set of convex functions: for $a\in \Rbar$, $f\in\Cvx(V,\Rbar)$ and $x\in V$,
 \[(a\cotensor f)(x)=f(x)-a;\quad (a\odot f)(x)=f(x)+a.\]
The product $\sqcap$ and cotensor $\cotensor$ make the set of convex functions $f\in\Cvx(V,\Rbar)$ into a module over the tropical semiring $(\Rbar, \min,+)$, that is $\Rbar$ with minimum as addition and usual addition as multiplication.  Similarly, the coproduct $\sqcup$ and tensor $\odot$ make the set of convex functions $\Cvx(V,\Rbar)$ into a module over the tropical semiring $(\Rbar, \min,+)$ in a different way.

\subsection{Synopsis}
The bulk of the paper is Section~2 in which we go through a lot of essentially standard category theoretic notions such as presheaves, adjunctions, profunctors, completeness and cocompleteness, but we do this in the non-standard context of categories enriched over the extended real numbers $\Rbar$; this is intended to be suitable for those interested in the Legendre-Fenchel transform with little background in category theory.  The nucleus of a profunctor is described in Section~3 by analogy with constructing a pair of adjoint linear maps from a matrix.  In Section~4 we see how the usual theory of Galois connections from relations arises by considering the case of categories enriched over classical truth values.  The pay-off comes in Section~5 in which the preceding theory is put together for $\Rbar$-categories and a lot of the theory of the Legendre-Fenchel transform drop out, such as Toland-Singer duality and the two tropical module structures on the set of convex functions.

\subsection{Further questions}  There are several questions that this work naturally leads to: here are two of them.
\begin{itemize}
\item
What happens if we work over $\Zbar=\Z\cup \{-\infty,+\infty\}$?  Discrete convex analysis seems to be more subtle (see~\cite{Murota:DiscreteConvexAnalysis} and~\cite{Fujii:BThesis}) and it would be worth seeing if this approach leads to anything interesting.
\item
Fenchel Duality (see Theorem~31.1 of~\cite{Rockafellar}) involves both convex and concave functions; is it possible to formulate this in categorical terms and does it generalize to other situations?
\end{itemize}

\subsection{Acknowledgements}
I would like to thank Tom Leinster, Mark Kambites, Marianne Johnson, Hiroshi Hirai, Soichiro Fujii and Bruce Bartlett for useful comments and conversations.  I would also like to thank contributors to the $n$-Category Caf\'e  where earlier pieces of this work was posted. 

\section{Categories enriched over posets}
\label{sec:posetenriched}
In this section the theory of categories enriched over monoidal posets is presented with particular emphasis placed on enriching over the posets $\Rbar$ and $\Truth$.  The intention is that this should be accessible to a reader with an interest in the Legendre-Fenchel transform but with not much background in category theory.
 \subsection{Quantales, thin categories and idempotent semirings}
In this paper we consider categories enriched over monoidal posets.  The structure of a monoidal poset crops up in different areas with different names and different features emphasised, so it is worth explaining here how these connect.   Our main examples --- the extended real numbers $\Rbar$ with the order $\ge$ and the set of classical truth values $\Truth$ with entailment as the order --- are given after the definitions.

\subsubsection{Commutative quantales}
In concise terms, a \define{commutative quantale} is a complete lattice which has a commutative, unital, associative, binary operation $\otimes$ which distributes over arbitrary joins.  Unpacking this definition, we have a poset%
\footnote{I am assiduously not preferring either $\le$ or $\ge$ as my relation.  Here $\relation$ can be pronounced as \scare{is related to}.}
 $(Q,\relation)$ such that every subset of $W\subset Q$ has both a meet $\bigwedgie_{x\in W}x$ and a join $\biggyvee_{x\in W}x$; in particular, these satisfy, for all $w\in W$, 
   \[\bigwedgie_{x\in W}x\relation w\quad \text{and}\quad w\relation \biggyvee_{x\in W}x.\]
When $\relation$ is $\le$ then the meet $\bigwedgie$ is the infimum, and when it is $\ge$ then the meet  $\bigwedgie$ is the supremum.

This poset is equipped with a monoid structure $\otimes\colon Q\times Q \to Q$ with unit $\one\in Q$ and which satisfies distributivity over joins:
  \[q\otimes \biggyvee_{x\in W} x=\biggyvee_{x\in W} (q\otimes x).\]
From this data  we can define a residuation map $[{-},{-}]\colon Q\times Q\to Q$ by 
  \[[b,c]\coloneqq \biggyvee_{a\colon a\otimes b\relation c} a.\]
 This will satisfy the \define{adjunction property}
   \[a\otimes b\relation c \quad \Longleftrightarrow \quad a\relation [b,c].\]

\subsubsection{Idempotent semirings}
A \define{commutative idempotent semiring} is a commutative semiring in which the addition is idempotent.
Firstly, a commutative semiring, basically a ring without negatives, consists of two commutative monoid structures $(S,\oplus,0)$ and $(S,\otimes,1)$ on the same set $S$ such that $\otimes$ distributes over $\oplus$.  The addition is idempotent if 
  \[s\oplus s=s\quad\text{ for all } s\in S.\]
This allows us to define a partial order on $S$ by $a\relation b$ if and only if $a\oplus b=b$.  Then $\oplus$ becomes the join $\biggyvee$, and $\otimes$ distributes over it.  By induction this means that we have  joins for arbitrary finite subsets, and that multiplication distributes over these finite joins.  We say that $S$ is \define{complete} if arbitrary subsets have joins and multiplication distributes over arbitrary joins.  Arbitrary \emph{meets} can then be defined in terms of the joins: for a subset $W\subset S$ define $W'\coleqq \{y\in S \mid y\relation x\,  \text{for all}\, x\in W\}$ then $\bigwedgie_{x \in W}x=\biggyvee_{y\in W'} y$.  Again the residuation can be defined by 
  \[[b,c]\coloneqq \bigoplus_{a\colon a\otimes b\relation c} a.\]
One can see  in this manner that a complete, commutative, idempotent semiring is the same thing as a commutative quantale.

\subsubsection{Monoidal thin categories}
A \define{thin category} is a category in which there is at most one morphism between each pair of objects.  Working in a thin category is much simpler than working in an arbitrary category as, for instance, all diagrams commute.  Having a thin category is the same as having a transitive, reflexive relation on the objects of the category, the correspondence is defined via
  \[a\relation b\quad\Longleftrightarrow\quad \text{there is a morphism } a\to b.\]
This is not quite the same thing as having a partial order, however, as anti-symmetry might fail: we could have distinct objects $a$ and $b$ which have arrows $a\to b$ and $b\to a$.  Such a relation is called a preorder.  The condition that a thin category actually corresponds to a partial order is precisely the \define{skeletal condition} which is that all isomorphisms are equalities.  We can turn any thin category into a skeletal category by taking isomorphism classes of objects.

A \define{monoidal category} is a category $\C$ with a functor $\otimes\colon\C\times \C\to \C$  together with some associativity and unitality constraints.  In the case of a skeletal, thin category the constraints are just that the product should be associative and unital on the set of objects.  The monoidal structure is \define{symmetric} if this product is commutative.  The monoidal structure is \define{closed} if there is a functor $[{-},{-}]\colon \V^\op\times \V\to \V$ which is adjoint to $\otimes$ in the sense that 
  \[\text{there is a morphism }a\otimes b\to c \quad \Longleftrightarrow \quad \text{there is a morphism }a\to [b,c].\]
As you should be able to tell, this is the same as the condition on the residuation for a quantale.  There are notions of categorical products and coproducts, for thin categories these amount to precisely meets and joins for the associated preorders.  As the monoidal product has a right adjoint it will automatically distribute over joins.    Being complete and cocomplete in the categorical sense amounts, for a thin category, to having arbitrary products and coproducts.

From this concise discussion it is hopefully clear that a complete and cocomplete, skeletal, closed, symmetric monoidal thin category is the same thing as a commutative quantale.

\subsection{Our main examples}  The main examples of the above structures that we are interested in here are truth values and extended real numbers.  We will get to Galois connections using the first and the Legendre-Fenchel transform using the second.

\subsubsection{$\Truth$}
We take $\Truth$ to be the category with two objects $\true$ and $\false$ and a single non-identity arrow $\false\to \true$.  The monoidal product is taken to be \scare{logical and} $\logand$, and the residuation is implication, so $[a,b]= ( a \Rightarrow b)$.  The order on the set $\{\true,\false\}$ can be interpreted as \scare{entailment} $\entails$, so the only non-reflexive relationship is $\false\entails\true$.  The meet is the same as the product, namely logical and, but is often thought of as \scare{for all}.  The join is \scare{logical or}, often thought of as \scare{there exists}.

\subsubsection{$\Rbar$}  
\label{sec:tables}
To define the structure on the extended real numbers, we start with $\R$ equipped with the order $\ge$, so that meet $\bigwedgie$ is supremum and join $\biggyvee$ is infimum, this has usual addition as the monoidal product and subtraction as the residuation because
  \[a+b\ge c\quad \iff \quad a\ge c-b.\]
We make this structure complete by adding a maximum element $+\infty$ and a minimum element $-\infty$ to define $\Rbar$.  Then the requirement that $+$ distributes over infima ensures that there is a unique extension of $+$ to $\Rbar$.  Finally we can use the formula for residuation to extend subtraction to $\Rbar$:
  \[c-b=\inf_{a+b\ge c}a.\]
The arithmetic is summarized in the following tables lifted from the Bachelor's Thesis of Fujii~\cite{Fujii:BThesis}. 
\begin{center}
    \newcolumntype{A}{>{$}c<{$}}
    \begin{tabular}{AA|AAA}
        &&&y&\\
        & x+y &-\infty&t&+\infty\\
        \otoprule
        &-\infty & -\infty &-\infty & +\infty\\
        x& s &-\infty & s+t & +\infty \\
        & +\infty & +\infty & +\infty & +\infty 
\end{tabular}
\qquad
    \begin{tabular}{AA|AAA}
        &&&y&\\
        &y-x&-\infty&t&+\infty\\
        \otoprule
        &-\infty & -\infty &+\infty & +\infty\\
        x&s &-\infty & t-s & +\infty \\
        &+\infty & -\infty & -\infty & -\infty 
    \end{tabular}
\end{center}
It is important to observe that there are a few subtleties in these definitions, so they don't behave quite how you might naively expect.  For instance, subtracting $+\infty$ is not the same as adding $-\infty$:
  \[(+\infty)+(-\infty)= (+\infty);\qquad (+\infty)-(+\infty)= (-\infty).\]
A similar structure, but with the opposite order was considered by Lawvere in~\cite{Lawvere:Entropy};  this structure was considered from the idempotent semiring perspective in~\cite{CohenGaubertQuadrat:DualityAndSeparation}.  Whilst this arithmetic of the infinite might seem a little strange, it is the appropriate structure for the Legendre-Fenchel transform.

As an idempotent semiring, $\Rbar$ has $+$ as its multiplication and $\min$ as its addition.
\subsection{Enriched structures for $\Truth$ and $\Rbar$}
We will go through a lot of the theory of categories enriched over our two main examples.  The standard references for much of the general theory of enriched categories are \cite{Kelly:EnrichedCategoryTheory}
and \cite{Borceux:Handbook2}.

\subsubsection{Categories}
\label{Sec:EnrichedCategories}
Recall that an ordinary small category $\C$ can be defined to consist of a set of objects $\ob(\C)$ such that
\begin{itemize}
 \item for every $c,c'\in \C$ there is a set $\C(c,c')\in \ob(\Set)$ called the hom-set;
 \item for every $c,c',c''\in \C$ there is a function $\C(c,c')\times \C(c',c'')\to \C(c,c'')$, this is called composition;
 \item for  every $c\in \ob\C$ there is a function $\{\pt\}\to \C(c,c)$, the image of this is called the identity on $c$.
\end{itemize}
This structure is required to satisfy unit and associativity conditions.

Given a monoidal category $\V$ we can define a \define{category enriched over $\V$} or a \define{$\V$-category} or just an \define{enriched category}, $\C$, to consist of a set $\ob\C$ such that
\begin{itemize}
 \item for every $c,c'\in \ob\C$ there is a specified object in $\V$, the hom-object, $\C(c,c')\in \ob(\V)$;
 \item for every $c,c',c''\in \ob\C$ there is a specified morphism in $\V$, compostion, $\C(c,c')\otimes \C(c',c'')\to \C(c,c'')$;
 \item for  every $c\in \ob\C$ there is a specified morphism in $\V$, the identity, $\{\pt\}\to \C(c,c)$.
\end{itemize}
This structure is required to satisfy unit and associativity conditions, but we won't need to know what these are.

When $\V$ is a thin monoidal category things simplify.  If we think of the object of $\V$ as a monoidal poset $(Q,\relation)$, then a \define{$\V$-category} is a set $C$ with the following data
\begin{itemize}
 \item for every $c,c'\in C$ there is a specified element  $\C(c,c')\in Q$;
 \item for every $c,c',c''\in C$ the relation $\C(c,c')\otimes \C(c',c'')\relation \C(c,c'')$ holds;
 \item for  every $c\in  C$ the relation $\one\relation \C(c,c)$ holds.
\end{itemize}
The unit and associativity conditions are automatically satisfied which is why we didn't need to know what they are.

Specializing to the two cases of interest, we find that a \define{$\Truth$-enriched category} is a set $R$ which for each pair of elements $r_1,r_2\in R$ we have a truth value $R(r_1,r_2)$ satisfying composition and identity conditions.  We interpret $R(r_1,r_2)$ as the truth value of a relation between between $r_1$ and $r_2$ and if this is true we write $r_1\Rrelation r_2$.  The composition condition means precisely that this relation is transitive and the identity relations means that the relation is reflexive.  If we use $\lefttruthval a\righttruthval$ to denote the truth value of $a$ then these become the following:
\[
    \lefttruthval r_1\Rrelation r_2\righttruthval \logand \lefttruthval r_2\Rrelation r_3\righttruthval 
        \entails \lefttruthval r_1\Rrelation r_3\righttruthval;
    \qquad 
    \true\entails\lefttruthval r_1\Rrelation r_1\righttruthval,\,\text{i.e.\ }r_1\Rrelation r_1.
\]
This relation, in general, will not be a partial order as it is not forced to satisfy antisymmetry, so it will just be a preorder.  

An \define{$\Rbar$-enriched category} is going to be a bit like a metric space with negative distances allowed.  It will consist of a set $X$ and for each $x,x'\in X$ we have a number $X(x,x')\in \Rbar$, but we will usually write this as $\dd(x,x')$ to emphasise that it is being thought of as like a distance.  This will have to satisfy the following two conditions.
\begin{itemize}
  \item for all $x,x'\in X$ we have $\dd(x,x')+\dd(x',x'')\ge \dd(x,x'')$;
  \item for all $x\in X$ we have $0\ge \dd(x,x)$.
\end{itemize}
The first condition is, of course, the triangle inequality.  The second condition looks slightly odd, but in conjunction with the triangle inequality it is straightforward to show that either $\dd(x,x)=0$, in which case we say that $x$ is a \define{finite} point, or else $\dd(x,x)=-\infty$ in which case we say that $x$ is an \define{infinite} point.  It is also straightforward to show that $x$ is an infinite point if and only if $\dd(x,x')=\pm \infty$ for all $x'\in X$.

 The first example of an $\Rbar$-space is $\Rbar$ itself in which  we have $\dd(x,x')=x'-x$, where for infinite values we mean the subtraction as defined in Section~\ref{sec:tables}.  Unsurprisingly, there are precisely  two infinite points: $-\infty$ and $+\infty$.

There is the notion of being isomorphic in a $\V$-category $\C$.  If $c,c'\in\Ob\C$ with $\one\relation \C(c,c')$ and $\one \relation \C(c',c)$ then we say that $c$ and $c'$ are \define{isomorphic} and write $c\isomorphic c'$.  Isomorphic objects cannot be distinguished in the category.  A $\V$-category in which any pair of isomorphic objects are actually equal is called a \define{skeletal} $\V$-category.

In the $\Truth$ case, as mentioned above, $r\isomorphic r'$ if $r\Rrelation r'$ and $r' \Rrelation r$, so a skeletal $\Truth$-category is precisely a poset.

In the $\Rbar$ case, if $x\isomorphic x'$ then precisely one of the following two things hold: either $\dd(x,x')=0=\dd(x',x)$ and both objects are finite, or else $\dd(x,x')=-\infty=\dd(x',x)$ and both objects are infinite.

We can now obtain a dictionary for translating from certain notions in category theory to notions in order theory and $\Rbar$-space theory.
Part of this dictionary is summarized in Table~\ref{table:truthdictionary}.  Here we give more detail.

\begin{sidewaystable}
\centering
\begin{tabular}{lll}
\toprule
category& preorder & $\Rbar$-space\\\otoprule
functor & order preserving map&distance non-increasing map\\\midrule
set of natural transformations  & domination relation & asymmetric $\sup$ metric on functions\\\midrule
internal hom object in Set & logical implication in Truth&asymmetric distance in $\Rbar$\\\midrule
presheaf & downward closed subset (downset) & short $\Rbar$-valued function\\\midrule
copresheaf & upward closed subset (upset) & short $\Rbar^{\op}$-valued function\\\midrule
category of presheaves & downsets ordered by inclusion & short $\Rbar$-functions with\\
&&\qquad maximal climb distance\\\midrule
category of opcopresheaves & upsets ordered by containment &short $\Rbar^{\op}$-functions with \\
&&\qquad maximal descent distance\\\midrule
category of presheaves& powerset of a set& $\Rbar$-functions with \\
\qquad on a set&\qquad ordered by inclusion&\qquad maximal climb distance\\\midrule
category of opcopresheaves& powerset of a set& $\Rbar^{\op}$-functions with\\
\qquad on a set&\qquad ordered by containment&\qquad maximal descent distance\\\midrule
adjunction & Galois connection&$\Rbar$-adjunction\\\midrule
profunctor & relation& $\Rbar$-valued pairing\\\midrule
nucleus of a profunctor& Galois correspondence from a relation&???\\\bottomrule
\end{tabular}
\caption{Translating from  categories to preorders and $\Rbar$-spaces}
\label{table:truthdictionary}
\end{sidewaystable}

\subsubsection{Functors}

There is a notion of \define{$\V$-functor} $F\colon \C\to\D$ between $\V$-categories.  This consists of a function $F\colon \ob \C\to \ob\D$ together with, for each pair of objects $c,c'\in \ob\C$, a $\V$-morphism $\C(c,c')\to \D(F(c),F(c'))$, such that these satisfy functoriality, but when $\V$ is thin this is automatically satisfied.

In the $\Truth$ setting a $\Truth$-functor is a function $F\colon R\to S$ such that,  
using square brackets to mean \scare{the truth value of} this is the same as \[\lefttruthval r\Rrelation r'\righttruthval\entails \lefttruthval F(r)\Srelation F(r')\righttruthval,\] or, in other words, $F$ preserves the order.

In the $\Rbar$ setting, an $\Rbar$-functor is a function $F\colon X\to Y$ such that 
 \[\dd(x,x')\le \dd(F(x),F(x'))\]
so an $\Rbar$-functor is a distance non-increasing map, also known as a \define{short map}.

Two functors $F\colon \C\rightleftarrows \D\colon G$ form an \define{equivalence} if $FG(d)\isomorphic d$ and $GF(c)\isomorphic c$ for all $c\in \C$ and $d\in \D$.  This is the right notion of sameness for categories.  If $\C$ and $\D$ are both skeletal then an equivalence between them is necessarily an isomorphism.

\subsubsection{Natural transformation objects}
In ordinary category theory, for a pair of parallel functors $F,G\colon \C\to \D$ there is a set of natural transformations between them.  This is sometimes written $\Nat(F,G)$ but category theorists often write it as $[\C,\D](F,G)$ because for such categories $\C$ and $\D$, it makes the collection of functors and natural transformations into a category, the functor category, written $\Fun(\C,\D)$ or $[\C,\D]$.  In the \emph{enriched} setting, providing that $\V$ is sufficiently nice, meaning complete, which we have in our cases, then for a pair of parallel functors $F,G\colon \C\to \D$ there is an object of $\V$, written $[\C,\D](F,G)$, known as the \define{$\V$-object of natural transformations}, and, again, for given $\V$-categories $\C$ and $\D$ we get a $\V$-category, the \define{functor $\V$-category} $[\C,\D]$, formed from the $\V$-functors from $\C$ to $\D$.  In general, the natural transformation object is given by a certain kind of limit known as an enriched end and is written as
$\int_c \D(F(c),G(c))$.
We  won't need ends in general as in the case of a thin $\V$ this is just a meet   \[[\C,\D](F,G)=\bigwedgie_{c\in \C} \D(F(c),G(c)).\]

In the context of enriching over $\Truth$, this becomes a big \scare{for all}%
\footnote{You might think of the meet as being \scare{and} but in the quantale perspective of generalizing truth values, it is the tensor product which generalizes logical \scare{and} whilst the meet generalizes \scare{for all}.}.
  \[[R,S](F,G)=\bigforall\nolimits_r \D(F(r),G(r)). \]
Rewriting this is terms of truth values of the relations we get that this means
  \[\lefttruthval F\RtoSrelation G\righttruthval \coloneqq \lefttruthval\forall r (F(r) \Srelation G(r))\righttruthval.\]
In other words, given preorders $\C$ and $\D$, there is a canonical preorder on the set of order preserving functions $\C\to\D$, and this is $F\RtoSrelation G$ if and only if $F(r)\Srelation G(r)$ for all $r$.  We can say that $F$ is \define{dominated} by $G$ if $F\RtoSrelation G$.  

In the context of enriching over $\Rbar$, the end is just a supremum, so the set $[X,Y]$ of all distance non-increasing maps $X\to Y$ is equipped with the natural $\Rbar$-metric
 \[\dd(F,G)\coloneqq\sup_c \{ \dd\left( F(c),G(c)\right)\}.\]
Of course, this is very similar to the sup metric on the function spaces between ordinary metric spaces.

\subsubsection{The underlying preorder of a $\V$-enriched category}
\label{sec:UnderlyingPreorder}
Associated to every enriched category $\C$ is an `underlying' ordinary category $\underlying{\C}$ which has the same objects as $\C$ but the hom sets are defined as follows: $\underlying{\C}(c,c'):=\V(\one,\C(c,c'))$.  This gives a \emph{set} because $\V$ is an ordinary category.  In the case that $\V$ is a thin category then the underlying category of any $\V$-category will also be thin.  To put it another way, if $\V$ is a preorder then $\V$-categories have underlying preorders.  For $\C$ a $\V$-category, the \define{underlying preorder} is seen, by unpacking the definition to be given by
  \[c\relation_{{\C}}c'\quad \iff\quad \one \relation \C(c,c').\]
It follows easily that the underlying preorder of $\C$ is a partial order precisely when $\C$ is skeletal.

In the case of $\Truth$, the underlying preorder of a $\Truth$-category is, rather unsurprisingly, the usual preorder that it is identified with.

In the case of an $\Rbar$-space it makes sense to write the preorder $\relation_{X}$ as $\curlyge$.  We have
    \[x\curlyge x'\quad \iff\quad 0\ge \dd(x,x').\]

\subsubsection{The category of $\V$-categories is closed monoidal}
\label{Section:VcatClosedMonoidal}
The ordinary category of $\V$-categories and $\V$-functors has the structure of a closed monoidal category.  This means that there is a tensor product of $\V$-categories and this has, in an appropriate sense, a right adjoint, which is actually the functor $\V$-category defined above.

If $\C$ and $\D$  are $\V$-categories then their \define{tensor product}  $\C\otimes \D$ has as its set of objects the set of ordered pairs $\ob \C \times \ob \D$ and the hom-$\V$-objects are given by 
  \[\C\otimes \D ((c,d),(c',d')):=\C(c,c')\otimes \D(d,d').\]

When enriching over truth values, this means that for preorders $R$ and $S$ the preorder on $R\times S$ is given by 
  \[(r,s)\RSrelation (r',s')\quad\text{ if and only if }\quad r\Rrelation r' \,\,\text{and}\,\, s\Srelation s'.\]

When enriching over the extended real numbers, this means that for $\Rbar$-spaces $X$ and $Y$ the $\Rbar$-metric on $X\times Y$ is given by
  \[\dd ((x,y),(x',y')):=\dd(x,x')+ \dd(y,y').\]
  
As described above we also have the functor $\V$-category $[\C,\D]$.  The notation is the same as for residuation because this is also adjoint to the tensor product in the sense that
  \[\Vcat(\C\otimes \D,\E)\isomorphic \Vcat(\C,[\D,\E])\]
so every $\V$-functor $\C\otimes\D\to \E$ corresponds to a unique $\V$-functor $\C\to [\D,\E]$.
%
%
%
%

If $\C$ and $\D$  are $\V$-categories then there is also their \define{cartesian product}  $\C\times \D$ which also has as its set of objects the set of ordered pairs $\ob \C \times \ob \D$ but the hom-$\V$-objects are given by using the meet in $\V$ rather than the tensor product:
  \[\C\times \D ((c,d),(c',d')):=\C(c,c')\wedge \D(d,d').\]
For $\Truth$-categories this is just the same as the tensor product, but for $\Rbar$-categories this is  an $L^{\infty}$-like product of spaces rather than the $L^{1}$-like tensor product. 

\subsubsection{Completeness and cocompleteness}
\label{sec:completeness}
  In ordinary category theory there are notions of limit and colimit which lead to the notions of completeness and cocompleteness for a category.  In enriched category theory there are notions of \emph{weighted} limit and colimit which lead to notions of completeness and cocompleteness for an enriched category.  We won't need the full theory as things simplify a lot when enriching over preorders, and we can characterize completeness and cocompleteness is a more direct way.  This is even simpler when the category under consideration is skeletal.

We say that a skeletal $\V$-category $\C$ \define{has products} if we have a function $\bigproduct\colon \powerset (\ob \C)\to \ob \C$ on the set of subsets of objects which satisfies
 \[\C\Bigl(c,\bigproduct_{c'\in S}c'\Bigr)=\bigwedgie_{c'\in S}\C(c,c')\qquad \text{for all }c\in \ob\C,\ S\subseteq\ob\C.\]
Such a function, if it exists, is necessarily unique as we assumed that $\C$ is skeletal.  This is because if we have a different function $\overline\bigproduct\colon \powerset (\ob \C)\to \ob \C$ with the same property, then we can use the identity on $\bigproduct_{c'\in S}c'$ to form the following composite:
  \[\one\to
  \C(\bigproduct_{c\in S}c,\bigproduct_{c'\in S}c')=
  \bigwedgie_{c'\in S}\C(\bigproduct_{c\in S}c,c')=
  \C(\bigproduct_{c\in S}c,\overline\bigproduct_{c'\in S}c').
  \]
By symmetry we also have a morphism  $\one\to
  \C(\overline\bigproduct_{c'\in S}c',\bigproduct_{c\in S}c).
$, thus $\overline\bigproduct_{c'\in S}c'$ and $\bigproduct_{c\in S}c$ are isomorphic, hence equal.

From the properties of the meet $\bigwedgie$ in $\V$ it follows that this gives the set of objects $\ob\C$ the structure of a commutative idempotent monoid in which you can take the product of infinite sets of elements.  The preorder associated to this idempotent structure is precisely the underlying preorder of $\C$.

We say that a skeletal $\V$-category $\C$ is \define{cotensored} if there is a function $\cotensor\colon\ob\V\times\ob\C\to \ob\C$ (pronounced ``pitchfork'' or ``cotensor'') such that 
 \[\C\bigl(c,v\pitchfork c'\bigr)=\bigl[v,\C(c,c')\bigr]\qquad \text{for all }c,c'\in \ob\C,\ v\in\ob\V.\]
If $\C$ has products then a simple calculation shows that being cotensored makes the monoid $(\ob\C,\bigproduct)$ into an $\V$-module.  Again, if the cotensor exists then it is necessarily unique because of the skeletal assumption on $\C$.

We say that a skeletal $\V$-category $\C$ is \define{complete} if it has products and is cotensored.
The poset $\V$ is itself complete because the product is given by the meet and the cotensor is given by residuation or internal hom: $v\pitchfork v' =[v,v']$.

In the case of skeletal $\Truth$-categories, i.e.~posets, categorical completeness is just the same as usual completeness in that the product is just the meet and the cotensor is given by the following: $\true\cotensor a=a$  and $\false\cotensor a=\bigwedgie_{\emptyset}$.

In the case of skeletal $\Rbar$-categories, by definition the product and cotensor satisfy 
  \[\dd\Bigr(a, \bigproduct_{x\in S} x\Bigl) =\sup_{x\in S}\dd(a,x)\qquad\text{and}\qquad\dd(a, r\cotensor x)= \dd(a,x)-r,\]
so, in particular, cotensoring with a positive number results in an object closer from every other point.

This definition of completeness is equivalent to the usual one in terms of existence of all weighted limits, this can be seen by applying the reasoning in Section~5 of~\cite{Willerton:Isbell} or using Theorem~6.6.14 of~\cite{Borceux:Handbook2}

A \define{continuous} functor between complete $\V$-categories is one which commutes with products and cotensors.

It is maybe worth noting that the hom objects in a complete $\V$-category are determined by the product and cotensor structure (see Section~2.2 of~\cite{CohenGaubertQuadrat:DualityAndSeparation}):
  \[\C(c,c')=\biggyvee \{v\in V \mid c\relation (v\cotensor c')\}.\]

In an analogous fashion, there are notions of coproducts, tensoring, cocontinuity and cocompleteness.
We say that a skeletal $\V$-category $\C$ \define{has coproducts} if we have a function $\bigcoproduct\colon \powerset (\ob \C)\to \ob \C$ on the set of subsets of objects which satisfies
 \[\C\Bigl(\bigcoproduct_{c'\in S}c',c\Bigr)=\bigwedgie_{c'\in S}\C(c',c)\qquad \text{for all }c\in \ob\C,\ S\subseteq\ob\C.\]
Again, such a function, if it exists, is necessarily unique.
And also it follows that this gives the set of objects $\ob\C$ the structure of a commutative idempotent monoid in which you can take the product of infinite sets of elements.

We say that a skeletal $\V$-category $\C$ is \define{tensored} if there is a function $\ctensor\colon\ob\V\times\ob\C\to \ob\C$  such that 
 \[\C\bigl(v\ctensor c',c\bigr)=\bigl[v,\C(c',c)\bigr]\qquad \text{for all }c,c'\in \ob\C,\ v\in\ob\V.\]
Again, if $\C$ has coproducts and then the tensoring makes the monoid $(\ob\C,\bigcoproduct)$ into a  $\V$-module.

In the obvious way, we say that a skeletal $\V$-category $\C$ is \define{cocomplete} if it has coproducts and is tensored.

\subsubsection{Presheaves and copresheaves}
A particularly important role is played in ordinary category theory by so-called presheaves and copresheaves; these are contravariant and covariant functors taking values in the category of sets and are analogues of scalar-valued functions on vector spaces.  We have an appropriate generalization in enriched category thoery provided that $\V$ is closed, symmetric monoidal, as it is in our cases, as then $\V$ can itself be considered as a $\V$-category.  A \define{presheaf} on a $\V$-category $\C$ is a functor $\C^\op\to \V$ and a \define{copresheaf} is a functor $\C\to \V$.  We can then form the functor $\V$-categories, so $\pre{\C}$ is the \define{presheaf} category $[\C^\op,\V]$ and $\copre{\C}$ is the \define{opcopresheaf category} $[\C,\V]^\op$.

In the $\Truth$ case we can identify these with more familiar concepts.  If $R$ is a $\Truth$-category thought of as a preorder then a presheaf is an order reversing function $P\colon R^\op\to \Truth$  and can be identified with the set $\tilde{P}\coleqq P^{-1}(\true)\subset R$.  The fact that $P$ is order reversing corresponds precisely to the property that $\tilde{P}$ is downward closed, so if $r\in \tilde{P}$ and $r'\Rrelation r$ then $r'\in\tilde{P}$.  Thus the set of presheaves on $R$ can be identified with the set of downward closed subsets of $R$.   The domination relation on presheaves then becomes the subset ordering on subsets of $R$:
  \begin{align*}
    P_1\le P_2&\iff \forall r. \left(P_1(r)\entails P_2(r)\right) \iff\forall r \left( r\in \tilde{P_1}\implies  r\in \tilde{P_2}\right)\\
    &\iff \tilde{P_1}\subset \tilde{P_2}.
  \end{align*}
  
%

%

Similarly, the domination relation on copresheaves corresponds to inclusion of the associated upward closed subsets.  However, it is usually the \emph{opposite} of the category of copresheaves that crops up, so this has the opposite relation.


Sets can be thought of as discrete posets, that is to say, where $r\le r'$ if and only if $r=r'$.  In that case all subsets are both upward closed and downward closed, thus the set of copresheaves and the set of presheaves can both be identified with the powerset of the original set.

In the case of $\Rbar$-spaces, a presheaf is simply a reverse-distance non-increasing function to $\Rbar$ so it satisfies
  \[\dd(x,x')\ge P(x)-P(x').\]
The $\Rbar$-distance on the set of such functions is the \scare{maximal climb} distance:
 \[\dd(P,P')=\sup_{x\in X}\left(P'(x)-P(x)\right).\]
The underlying partial order (see Section~\ref{sec:UnderlyingPreorder}) is then easily seen to be a domination relation:
 \[P\curlyge P'\quad\iff\quad P(x)\ge P'(x) \text{ for all }x\in X.\]

Similarly a copresheaf is a distance non-increasing function to $\Rbar$ so it satisfies
  \[\dd(x,x')\ge Q(x')-Q(x).\]
The $\Rbar$-distance on the set of opcopresheaves is the \scare{maximal fall} distance:
 \[\dd(Q,Q')=\sup_{x\in X}\left(Q(x)-Q'(x)\right).\]
 
\subsubsection{Presheaves and opcopresheaves are complete and cocomplete} 
\label{Section:PresheavesComplete}
The $\V$-category of presheaves, $\pre{\C}$, on a $\V$-category $\C$ is both complete and cocomplete, this is because $\V$ is assumed to be complete and cocomplete so the limits and colimits are defined pointwise:
 \begin{align*}
 \bigl(\bigproduct_{j\in J}P_j\bigr)(c)&=\bigwedgie_{j\in J} P_j(c);
 &
  \bigl(\bigcoproduct_{j\in J}P_j\bigr)(c)&=\biggyvee_{j\in J} P_j(c);
 \\
 (v\cotensor P)(c)&=v\cotensor(P(c));
 &
  (v\odot P)(c)& =v\odot (P(c)).
 \end{align*}

The category of presheaves $\pre{C}$ has a universal property for maps from $\C$ to cocomplete $\V$-categories.  Firstly, we need to observe that there is the Yoneda map $\Yoneda\colon\C\to \pre{\C}$ given by $c\mapsto \C({-},c)$.  Then, given a functor $F\colon\C\to \D$ with $\D$ cocomplete (so in particular skeletal) there is a unique cocontinuous functor $\pre{F}\colon\pre{\C}\to \D$ which factorizes $F$ through the Yoneda map.  For general $\V$-categories this functor has a description as a coend, but in the case here of enriching over a poset it can be written as a coproduct of tensors:
  \[\pre{F}(P)= \bigcoproduct_{c\in \C}P(c)\odot F(c).\]  
Because of this universal property, the category of presheaves is known as the \define{free cocompletion} of $\C$; from the $\V$-module perspective, we can think of it as the free cometric $\V$-module on $\C$.  If we have a function $f\colon C\to D$ from a set $C$ to a module $D$ over a ring $V$ then this extends to a $V$-module map from the $V$-module of $V$-valued functions on the set by $\pre{f}(p)=\sum_c p(c)f(c)$.

In the $\Truth$ case, for a poset $R$, thinking of presheaves as downsets, the Yoneda embedding $R\to \pre{R}$ sends an element to its descending set $r\mapsto \{r'\mid r'\le_{R}r\}$.  Then given an order preserving map $f\colon R\to S$ to a complete poset $S$ the unique cocontinuous extension $\pre{f}\colon \pre{R}\to S$ is given by $\pre{f}(D)=\biggyvee_{r\in D}f(r)$.

In the $\Rbar$ case, the Yoneda embedding $X\to \pre{X}$ associates to a point, the distance to the point $x\mapsto \dd({-},x)$.   

Similarly, the $\V$-category $\opcopre{\C}=[\C,\V]^{\op}$ of opcopresheaves is complete and cocomplete, with the limits and colimits defined pointwise as follows:
 \begin{align*}
 \bigl(\bigproduct_{j\in J}Q_j\bigr)(c)&=\biggyvee_{j\in J} Q_j(c);
 &
  \bigl(\bigcoproduct_{j\in J}Q_j\bigr)(c)&=\bigwedgie_{j\in J} Q_j(c);
 \\
 (v\cotensor Q)(c)&=v\otimes(Q(c));
 &
  (v\odot Q)(c)& =[v, Q(c)].
 \end{align*}

 The category of opcopresheaves $\pre{C}$ has a universal property for maps from $\C$ to \emph{complete} $\V$-categories.  This time we have the co-Yoneda map $\C\to \opcopre{\C}$ given by $c\mapsto \C(c,{-})$.  Then, given a functor $F\colon\C\to \E$ with $\E$ cocomplete (so in particular skeletal) there is a unique continuous functor $\opcopre{F}\colon\pre{\C}\to \D$ which factorizes $F$ through the Yoneda map.  It can be written as a product of cotensors:
  \[\pre{F}(P)= \bigproduct_{c\in \C}P(c)\cotensor F(c).\]
Because of this universal property, the category of opcopresheaves is known as the \define{free completion} of $\C$.

\subsubsection{Adjunctions}
A fundamental concept in category theory is that of an adjunction and a basic slogan of category theory is ``Adjunctions arise everywhere''~\cite[preface]{MacLane:CWM}.  Enriched adjunctions also arise in many places as we shall see.

An \define{enriched adjunction} consists of a pair of enriched functors
  \[F\colon \C \leftrightarrows \D\colon G\]
together with isomorphisms in $\V$, natural in $c\in \C$ and $d\in \D$
  \[\D(F(c),d)\cong \C(c,G(d)).\]
In our cases, $\V$ is skeletal so these isomorphisms will be equalities.

This means that when we enrich over the category of truth values we get a $\Truth$-adjunction being a pair of order-preserving maps between posets
  \[F\colon \C \leftrightarrows  \D\colon G\]
with the condition that 
  \[F(c)\le d \,\text{ if and only if }\, c\le G(d).\]
In other words we have that a $\Truth$-adjunction is precisely a Galois connection between preorders.

In the $\Rbar$ case a adjunction consists of a pair of distance non-increasing maps
    \[F\colon X \leftrightarrows  Y\colon G\]
with the condition that 
  \[\dd\left(F(x), y \right) = \dd\left( x, G(y)\right).\]
Here is a simple non-trivial example.  Take  $\Rbar$ with its usual $\Rbar$-metric $\dd(y_1,y_2)\coloneqq y_2-y_1$ and  take the cartesian product $\Rbar\times\Rbar$, that is $\Rbar^2$ with the $\Rbar$-metric 
\[\dd((x_1,x_2),(x_1',x_2'))\coloneqq \max(x_1'-x_1,x_2'-x_2).\]  Define \[F\colon \Rbar^2\to \Rbar;\quad F(x_1,x_2)=\min(x_1,x_2);\qquad G\colon \Rbar\to \Rbar^2;\quad G(y)\coloneqq (y,y).\]
Then $F$ and $G$ form an adjunction.


\subsubsection{Closure operations and the invariant part of an adjunction}
Associated to an adjunction $F\colon \C \leftrightarrows \D\colon G$ is an equivalence between certain subcategories of $\C$ and $\D$, this is known as the invariant part of the adjunction.  In order to describe this, consider first the composites $GF\colon \C\to\C$ and $FG\colon \D\to\D$.  Because we are working over a poset these are both idempotent  in the sense that $FGFG(d)\isomorphic FG(d)$ and $GFGF(c)\isomorphic GF(c)$ for all $c\in \C$ and $d\in\D$.  To demonstrate the isomorphism (see Section~\ref{Sec:EnrichedCategories}) for $FG$ we first use the unit map and then the adjunction:
  \[\one\to \C(GFG(d),GFG(d))\to \C(FGFG(d),FG(d)). \]
Then in the other direction we use functoriality in the final step,
  \[\one \to \D(FG(d),FG(d))\to \C(G(d),GFG(d))\to \D(FG(d),FGFG(d)).\]
The isomorphism for $GF$ is similar.  The examples we will be considering are skeletal categories so all of the isomorphisms can be replaced with equalities.

We also have a map in $\V$
 \[\one\to \C(G(d),G(d))\isomorphic \D(FG(d),d)\]
so in the underlying preorder $FG$ is a decreasing map: $FG(d)\relation_{\D}d$.  Similarly $GF$ is increasing: $c\relation_{\C}GF(c)$.  Together with the idempotency, this leads us to describe $FG$ and $GF$ as \define{closure operators}.  The \define{closed objects} will be the elements of the fixed categories of these operators, where by `fixed' we mean fixed up to isomorphism:
\[
\Fix(GF)\coloneqq \{c\in \C\mid c\isomorphic GF(c)\};\qquad
\Fix(FG)\coloneqq \{d\in \D\mid FG(d)\isomorphic d\}
\]
By the idempotency, these fixed categories are equivalent to the image categories, i.e.~$\Fix(FG)\equivalent\Image(FG)$ and $\Fix(GF)\equivalent\Image(GF)$.  The adjunction restricts to an equivalence between the fixed sets,
  \[\Fix(GF)\equivalent\Fix(FG).\]
There is a third equivalent category which consists of pairs of objects which are sent to each other by the adjunction:
  \[\Inv(F,G)\coleqq \bigl\{(c,d)\mid F(c)\isomorphic d,\ c\isomorphic G(d)\bigr\}\subset \C\times \D.\]
The equivalences $\alpha\colon \Inv(F,G)\leftrightarrows\Fix(GF)\colon \beta$ are given by $\alpha(c,d)\coleqq c$ and $\beta(c)\coleqq  (c,F(c))$, and it works analogously for $\Fix(FG)$.
Any of these three equivalent $\V$-categories, or the  equivalence between the fixed sets, is known as the \define{invariant part} of the adjunction.  
If the categories are skeletal then they must be an isomorphic.

In the case of $\Truth$ the functors $FG$ and $GF$ are the usual closure operators for a Galois connection and the invariant part is precisely the associated Galois correspondence.

In the $\Rbar$-adjunction given above, the invariant part identifies $\Rbar$ with the diagonal in $\Rbar\times\Rbar$.

\subsubsection{The invariant part is complete and cocomplete}

For an adjunction $F\colon \C \leftrightarrows \D\colon G$, if $\C$ and $\D$ are complete and cocomplete then so is the invariant part of the adjunction, with the limits being calculated in $\C$ and the colimits in $\D$ in the following way.

\begin{thm}
\label{thm:InvariantPartComplete}
If $\C$ and $\D$ are complete and cocomplete $\V$-categories then so is the invariant part of any $\V$-adjunction $F\colon \C\leftrightarrows \D\colon G$, with the limits and colimits given as follows: 
\begin{align*}
  \bigproduct_{j\in J}(c_j,d_j)
  &=
  \biggl(\bigproduct_{j\in J} c_j,\ 
  FG \Bigl(\bigproduct_{j\in J} d_j\Bigr)\biggr);
  \\
  \bigcoproduct_{j\in J}(c_j,d_j)
  &=
  \biggl(GF\Bigl(\bigcoproduct_{j\in J} c_j\Bigr),\ 
    \bigcoproduct_{j\in J} d_j\biggr);
  \\
  v\cotensor (c,d)
  & =
  (v\cotensor c ,\  FG (v\odot d));
  \\
  v\odot(c,d) 
  &=
  (GF(v\odot c),\ v\odot d).
\end{align*} 
\end{thm}
\begin{proof}
  We will prove the formula for products and work initially in $\Fix(GF)\in \C$.  As $\C$ is complete every subset $\{c_{j}\}\subset \Fix(GF)$ has a product $\bigproduct_{j} c_{j}\in \C$, however we will show that the product actually lies in $\Fix(GF)$, i.e.~$GF(\bigproduct_{j} c_{j})=\bigproduct_{j} c_{j}$, for the latter it suffices to show that there are morphisms $\one \to \C(\bigproduct c_{j}, GF(\bigproduct c_{j}))$ and $\one \to \C(GF(\bigproduct c_{j}), \bigproduct c_{j})$.
  
  The first of these morphisms can be obtained as the composite of the unit and the adjunction:
\[
  \one
  \to 
  \C\Bigl(F(\bigproduct c_{j}),F(\bigproduct c_{j})\Bigr)
  \isomorphic
  \C\Bigl(\bigproduct c_{j},GF(\bigproduct c_{j})\Bigr).
\]
For the second morphism we use the unit morphism with the definition of the product in $\C$ and the fact that $GF(c_{i})=c_{i}$:
\begin{align*}
  \one
  &\to
  \C\Bigl(\bigproduct c_{j},\bigproduct c_{i}\Bigr)
  =
  \bigwedgie_{i} \C\Bigl( \bigproduct c_{j},c_{i}\Bigr)
  \to
  \bigwedgie_{i} \C\Bigl( GF\bigl(\bigproduct c_{j}\bigr),GF(c_{i})\Bigr)\\
  &=
  \bigwedgie_{i} \C\Bigl( GF\bigl(\bigproduct c_{j}\bigr),c_{i}\Bigr)
  =
  \C\Bigl( GF\bigl(\bigproduct c_{j}\bigr),\bigproduct c_{i}\Bigr).
\end{align*}
These two morphisms mean that the two objects are isomorphic, $GF\bigl(\bigproduct c_{j}\bigr)\isomorphic \bigproduct c_{j}$, so by the skeletal condition on complete categories, they are actually equal, thus $\bigproduct c_{j}\in \Fix(GF)$ and it is the product in $\Fix(GF)$.

As $\Fix(GF)$ is equivalent to $\Inv(F,G)$, the products corresponds under the equivalence.   Writing, for clarity, $\bigproduct^{\C}$ to mean the product in the category $\C$, and recalling from above the equivalences $\alpha\colon \Inv(F,G)\leftrightarrows\Fix(GF)\colon \beta$, we find
 \begin{align*}\textstyle
 \bigproduct^{\Inv(F,G)} (c_{j},d_{j})
 &=
 \beta\bigl(\bigproduct^{\Fix(GF)} \alpha(c_{j},d_{j})\bigr)
 =
  \beta\bigl(\bigproduct\nolimits^{\Fix(GF)} c_{j}\bigr)
  =
  \beta\bigl(\bigproduct\nolimits^{\C} c_{j}\bigr)
  \\
  &=
  \Bigl(\bigproduct^{\C} c_{j},F \bigl(\bigproduct\nolimits^{\C} c_{j}\bigr)\Bigr)
 =
  \Bigl(\bigproduct^{\C} c_{j},F \bigl(\bigproduct\nolimits^{\C} G(d_{j})\bigr)\Bigr)\\
 &=
  \Bigl(\bigproduct^{\C} c_{j},FG \bigl(\bigproduct\nolimits^{\D} d_{j}\bigr)\Bigr).
\end{align*}
This is the required formula for the product.  The proofs of the other formulas is analagous.
\end{proof}
We can rephrase the theorem in another way: to take a limit --- i.e.~product or cotensor --- in the subcategory of closed objects in $\C$ you just take the limit in $\C$; whereas to take a colimit --- i.e.~coproduct or a tensor --- in the subcategory of closed objects  you take the \emph{closure} of the colimit in $\C$.

The category theoretic reader might note that the theorem is saying, amongst other things, that the inclusion $\Fix(GF)\to \C$ creates limits.  As $GF$ is idempotent, $\Fix(GF)$ is the Kleisli category $\C^{GF}$ of the monad $GF$ and the inclusion map $\Fix(GF)\to \C$ is actually the forgetful functor $\C^{GF}\to \C$, which in usual category theory is known to create limits.

A careful reader might note that for the invariant part to be complete and cocomplete it suffices that $\C$ is complete and $\D$ is cocomplete.  In that case the formulas in the theorem need changing slightly; for instance, we need to write $F \bigl(\bigproduct_{j\in J} G(d_j)\bigr)$ instead of $FG \bigl(\bigproduct_{j\in J} d_j\bigr)$, as the product $\bigproduct_{j\in J} d_j$ doesn't necessarily exist.

\subsubsection{Profunctors}
A \define{profunctor} between two $\V$-categories $\C$ and $D$ is  a $\V$-functor 
  \[\I \colon \C^\op\otimes \D\to \V.\]    
Whilst we won't need it here, it is worth saying a little more about profunctors, which are also called distributors, bimodules and matrices.  Conventions differ amongst authors, but I would consider such a profunctor to go from $\C$ to $\D$ and write $\I\colon \C\prof \D$.
Profunctors can be composed: if we have $\I\colon \C\prof \D$ and $\J\colon \D\prof \E$ then define the composite $\J\circ\I\colon \C\prof\E$ by \[\J\circ\I(c,e)\coloneqq \biggyvee_{d}\I(c,d)\otimes \J(d,e).\]

In the $\Truth$ case a profunctor $\I\colon R^{\op}\otimes S\to \Truth$ can in the usual way be identified with an upward closed subset $\widetilde\I\subset\R^{\op}\otimes \D$, we  can write $c\preceq_{\I}d$ for $(c,d)\in \tilde\I$, i.e.\ for $\I(c,d)=\true$ we have
  \[\text{if}\,\,c\preceq_{\I} d\,\,\text{and}\,\,(c,d)\le_{R^{\op}\otimes S}(c',d')\,\,\text{then}\,\,c'\preceq_{\I}d'.\]
This can make more sense if we expand it out:
  \[\text{if}\,\,c'\le_{R}c\preceq_{\I}d \le_{S} d' \,\,\text{then}\,\,c'\preceq_{\I} d'.\]
So we can say that a profunctor between preorders corresponds to preorder on $R\cup S$ extending the existing preorders.

In the $\Rbar$ case, the profunctor we are primarily interested in is the pairing between a real vector space and its dual.  (See below.)

\section{The nucleus of a profunctor}
In this section I will explain the construction of the nucleus of a profunctor: how, from a $\V$-profunctor $\M\colon \C^{\op}\otimes\D\to \V$, you get an adjunction $\M^{\ast}\colon \pre{\C}\leftrightarrows\opcopre{\D}\colon \M_{\ast}$ and then take its invariant part.  I will try to do this in a way that makes it clear that it is analogous to a basic linear algebra construction where to an $m\times n$ matrix $M$ you construct a pair of adjoint linear maps $M^{\ast}\colon \R^{m}\leftrightarrows\R^{n}\colon M_{\ast}$.  The first section on linear algebra can be skipped if you just want to get to the definition of the nucleus.

\subsection{Matrices and adjoint linear maps}
The notion of profunctor may be thought of as a categorification of the notion of a matrix.  In this section we will look at the decategorified version of the nucleus --- when categorifying it often helps to think carefully about the concept that is being categorified.  This section is not essential for what follows.

Given two finite sets $A$ and $B$ --- often taken to be $\{1,\dots ,n\}$ and $\{1,\dots, m\}$ for some natural numbers $n$ and $m$ --- a matrix can be understood as a function
   \[M\colon A\times B\to \R.\]
As all undergraduates understand --- or so we hope --- matrices are intimately related to linear maps between Euclidean spaces.  In particular, you can produce two linear maps
  \[M_\ast \colon \R^B\to \R^A,\qquad M^\ast \colon \R^A\to \R^B.\]
If you have an order on $A$ and on $B$, for instance, if they are $\{1,\dots ,n\}$ and $\{1,\dots, m\}$, then often you think of $\R^A$ and $\R^B$ as columns of numbers, and $M^\ast$ and $M_\ast$ as rectangular arrays of numbers, from this perspective, $M^\ast$ and $M_\ast$ will be mutually transposed arrays of numbers.  On top of this, $\R^A$ and $\R^B$ both have canonical inner product structures on them and with respect to these, the two linear maps are adjoint:
   \[\langle M^\ast p,q\rangle_{\R^B} = \langle p, M_\ast q\rangle_{\R^A},\qquad \text{for all }p\in \R^A, q\in\R^B\]

We should describe how we obtain $M_\ast$ and $M^\ast$ from $M$.  There are two standard ways to think about the vector space $\R^A$, one is as formal linear combinations of elements of $A$ and the other is as real-valued functions on the set $A$.  We will think in terms of functions, and so if $p\in \R^A$ is a vector and $a\in A$ then we will write $p(a)$ for the $a$ component of the vector.  As undergraduates, we would more likely have written $p_a$ for this.

For the first map, starting with the matrix $M\colon A\times B\to \R$, we take its adjoint in the category of sets to get a function $B\to\Set(A,\R)=  \R^A$ given by $b\mapsto(a\mapsto M(a,b))$.  We then use the fact that $\R^B$ is the free vector space on $B$ so there is a unique linear extension of that function to $M_\ast \colon \R^B\to \R^A$.  Explicitly, this linear extension is written as follows
  \[(M_\ast q)(a)\coloneqq \sum_{b\in B}M(a,b) q(b).\]

Similarly, for the second map we perform two steps.  Firstly we start with $M$ and take the adjoint function $A\to \R^B$ given by $a\mapsto(b\mapsto M(a,b))$.  Then we extend this by linearity to get a linear function $M_\ast \colon \R^A\to \R^B$ which is given by
  \[(M_\ast p)(b)\coloneqq \sum_{a\in A}M(a,b) p(a).\]

The vector spaces $\R^A$ and $\R^B$ have canonical inner products on them, namely, for $p,p'\in \R^A$ we have 
  \[  \langle  p,p'\rangle_{\R^A}\coloneqq \sum_{a\in A}p(a) p'(a).\]
These two functions we have constructed are adjoint as
 \[\langle M^\ast p,q\rangle_{\R^B} =\sum_{a\in A,b\in B}M(a,b)p(a) q(b)= \langle p, M_\ast q\rangle_{\R^A}.\]

That is all reasonably standard undergraduate material, however the analogue of  what we will want to do in the nucleus construction is to now take the fixed set of the adjunction.  This can be thought of  $\Fix(M^\ast M_\ast)\subset \R^B$ or as $\Fix(M_\ast M^\ast)\subset \R^A$ or as $\{(p,q)\in \R^A\times \R^B \mid M^\ast p= q\text{ and } p=M_\ast q\}$, these are all isomorphic spaces which can be described as the singular vectors of $M$ with singular value $1$.  I don't know any good reason for considering such a space associated to a matrix $M$, but I suspect there will be one.

\subsection{The nucleus construction}
\label{Section:NucleusConstruction}
We now see the nucleus constructin of a profunctor in a fashion so that it appears as the categorification of the above matrix construction.  Actually, in the case that the enriching category $\V$ is a poset, which we consider here, it is only really half a rung on the categorical ladder.

Fix the category $\V$ to be a complete, cocomplete, skeletal, closed symmetric monoidal category; our standard examples are $\Rbar$ and $\Truth$.  Also fix a profunctor $\M\colon \A^\op \otimes \B \to \V$.   We construct an adjunction between presheaf and opcopresheaf categories; the invariant part of this adjunction will be the nucleus of $\M$.

We proceed as we did for matrices.  Firstly, as the category $\Vcat$ of $\V$-categories is closed monoidal (see Section~\ref{Section:VcatClosedMonoidal}) with respect to the tensor product $\otimes$, the profunctor $\M\colon \A^\op\otimes \B\to \V$ has an adjoint
  \[\B\to [{\A^\op},\V]=\presheaf{\A}\]
given by $b\mapsto  \M({-},b)$.  As the presheaf category $\presheaf{\A}$ is complete (see Section~\ref{Section:PresheavesComplete}), we can take the unique continuous extension to the free completion of  $\B$, to get the continuous functor $\M_\ast \colon \opcopre{\B}\to \presheaf{\A}$.  
We can write this using the cotensor structure as
 \[(\M_\ast Q)(a) = \bigwedgie_b (Q(b)\cotensor\M(a,b)).\]
This should be compared with the comparable linear algebra formula: \[(M_\ast q)(a)\coloneqq \sum_{b\in B} M(a,b)q(b).\]

When $\V$ is $\Rbar$ we get 
  \[(\M_\ast Q)(a) = \sup_b \{\M(a,b)- Q(b)\},\]
and when $\V$ is $\Truth$ we get
  \[(\M_\ast Q)(a) = \lefttruthval \forall b\bigl( Q(b)\Rightarrow \M(a,b)\bigr)\righttruthval.\]
or in terms of the associated down sets:
  \[\widetilde {\M_\ast Q}= \left\{a\in \A\mid a\M b\text{ for all } b\in \tilde{Q}\right\}.\]
Similarly for $\M^\ast $,  we take the adjoint of the profunctor $\M\colon \A^\op\otimes \B\to \V$ to get a functor $\A^\op\to [\B,\V]$ which is the same as a functor $\A\to [\B,\V]^\op=\opcopre{\B}$ given by $a\mapsto \M(a,{-})$.  As the presheaf category $\opcopre{\B}$ is cocomplete, we can take the unique cocontinuous extension to the presheaves category on $\A$, to get the cocontinuous functor $\M^\ast \colon  \presheaf{\A}\to \opcopre{\B}$.  The formula for this will use tensors and coproducts in $\opcopre{\B}= [\B,\V]^\op$ which are the same as cotensors and products in $ [\B,\V]$ which we know are given pointwise: thus  
  \[(\M^\ast P)(b) = \bigwedgie_a P(a)\cotensor\M(a,b).\]
Again this needs comparing with the comparable formula in linear algebra: $(M_\ast q)(a)\coloneqq \sum_{b\in B} M(a,b)q(b)$.

Now we have a cocontinuous functor $\M^\ast \colon \presheaf{\A}\to \opcopre{\B}$ and a continuous functor $\M_\ast \colon \opcopre{\B}\to \presheaf{\A}$.  The former is left adjoint to the latter:
 \[\opcopre{\B}(\M^\ast P,Q)\cong\presheaf{\A}(P,\M_\ast Q).\]
Both sides are isomorphic to $\bigwedgie_{a,b}\bigl(P(a)\otimes Q(b)\bigr)\cotensor\M(a,b)$.  This should be compared with the linear algebra case.

%

The nucleus $N(\M)$ of the profunctor $\M\colon \A^\op\otimes\B\to\V$ is then defined to be the invariant part of the adjunction $\M^\ast\colon \pre\A\leftrightarrows \opcopre{\B}\colon\M_\ast$.  This means that the nucleus can be thought of as any of the following three isomorphic  $\V$-categories:
\begin{itemize}
\item $\Fix(\M_\ast\M^\ast)\subset\pre{\A}$; 
\item $\Fix(\M^\ast\M_\ast)\subset\opcopre{\B}$;
\item $\bigl\{(P,Q)\mid \M^\ast(P)=Q,\ P=\M_\ast(Q)\bigr\}\subset  \pre{\A}\times \opcopre{\B}$. 
\end{itemize}
As both $\pre{\A}$ and $\opcopre{\B}$ are complete and cocomplete (see Section~\ref{Section:PresheavesComplete}), we know, by Theorem~\ref{thm:InvariantPartComplete}, that the nucleus is also both complete and cocomplete, we also know formulas for the limits and colimits.

\section{The Galois correspondence from a relation}
As mentioned in the introduction, in the case of $\Truth$, the nucleus of a profunctor construction recovers the classical Galois correspondence from a relation construction.  Indeed, in the context of formal and fuzzy context analysis~\cite{Belohlavek:FuzzyGaloisConnections,Pavlovic:FCA} this is what the nucleus of a profunctor was designed to generalize.

A set $R$ can be considered as a discrete preorder, where the relation is the identity.  This is the same thing as being a discrete $\Truth$ category, so that $R(r,r')=\false$ if $r\ne r'$ and $R(r,r)=\true$  for all $r,r'\in R$.  A relation $\I$ between two sets $G$ and $M$ can thus be considered as a profunctor between the associated discrete truth categories.  The poset of presheaves on $G$ can be identified with the powerset of $G$ ordered by subset inclusion, while the poset of opcopresheaves on $M$ can be identified with the powerset of $M$ ordered in the opposite way.

Then the adjunction associated to the profunctor $\I$ can be viewed as a Galois connection
  \[\I^{\ast}\colon \PP(G)\leftrightarrows \PP(M)^{\op}\colon \I_{\ast},\]
which is, unsurprisingly, the classical Galois connection associated to the relation.  So taking the invariant part of this recovers the classical duality or Galois correspondence between the closed subsets of $G$ and the closed subsets of $M$:
  \[\PP_{\cl}(G)\cong \PP_{\cl}(M)^\op.\]
These are both complete lattices by the general completeness results of Theorem~\ref{thm:InvariantPartComplete}.  For instance, the meet $\bigwedgie$ in $\PP_{\cl}(G)$ is the meet in $\PP(G)$ whereas the join $\biggyvee$ in $\PP_{\cl}(G)$ is the \emph{closure} of the join in  $\PP(G)$.

\section{The Legendre-Fenchel transform}
\label{sec:LF}

In this section we will construct the nucleus in the context of the tautological pairing between a vector space and its dual; this will lead us to the theory of the Legendre-Fenchel transform and we will see how various properties  arise as general category theoretic properties.

If $V$ is a real vector space then we can consider it as a discrete $\Rbar$-space, so $\dd(x,x')=+\infty$ if $x\neq x'$ and $\dd(x,x)=0$.  Similarly we can consider the dual space $\dual{V}$ to be a discrete $\Rbar$-space.  The tautological pairing $\langle{-},{-}\rangle\colon V\times \dual{V}\to \R\subset \Rbar$ can then be considered to be a profunctor.  It is easy to check that it is a distance non-increasing map, where $\Rbar$ has its standard $\Rbar$ metric: $\dd(a,a')=a'-a$ for $a$ and $a'$ finite.

   The $\Rbar$-space $\pre{V}$ of presheaves  on $V$ is immediately identifiable with the space of functions $\Fun(V,\Rbar)$ equipped with the $\Rbar$-metric
\begin{align*}
    \dd(f_1,f_2)&\coloneqq\sup_{x\in V}\{f_2(x)-f_1(x)\}\quad \text {for all}\ f_1,f_2\colon V\to \Rbar.
\end{align*}
The distance from a function $f_1$ to a function $f_2$ is the most that you have to go \scare{up} from $f_1$ to get to $f_2$, where going down means going \scare{up} a negative amount.   It is worth observing that the infinite points, i.e.~those functions $f$ for which $\dd(f,f)=-\infty$, are precisely the functions which take value $\pm \infty$ at some point in $V$.

Similarly we can identify the $\Rbar$-space $\opcopre{\dual{V}}$ of opcopresheaves on the dual of $V$ with the space of functions $\Fun(\dual{V},\Rbar)$ equipped with the opposite $\Rbar$-metric:
  \begin{align*}
  \dd(g_1,g_2)&\coloneqq\sup_{k\in \dual{V}}\{g_1(k)-g_2(k)\}\quad \text {for all}\,\,g_1,g_2\colon \dual{V}\to \Rbar.
  \end{align*}

If we now begin to construct the profunctor nucleus of the tautological pairing then the first thing that we get is an $\Rbar$-adjunction between function spaces:
  \[\Fun(V,\Rbar)\leftrightarrows\Fun({\dual{V}},\Rbar).\]
Here, we will move away from our previous convention for notating these maps and denote them both in the more standard way of $f\mapsto \fench{f}$.  We find that both maps have the same form, so  for $f\colon V\to \Rbar$ a function on $V$ and $g\colon \dual{V}\to \Rbar$ a function on $\dual{V}$
 we have
 \begin{align*}
 \fench{f}(k)\coloneqq \sup_{x\in V}\big\{\langle x,k\rangle -f(x)\big\};\qquad
 \fench{g}(x)\coloneqq \sup_{k\in \dual{V}}\big\{\langle x,k\rangle -g(k)\big\}.
 \end{align*}
These are precisely the classical Legendre-Fenchel transform $\fench{{}}\colon \Fun(V,\Rbar)\to\Fun({\dual{V}},\Rbar)$ and its reverse $\fench{{}}\colon \Fun({\dual{V}},\Rbar)\to\Fun(V,\Rbar)$.

The fact that these maps form an adjunction means that \(\dd(\fench{f},g)=\dd(f,\fench{g})\).  Spelling this out, we have the following theorem.
\begin{thm}
\label{thm:LFadjunction}
Suppose that we have functions $f\colon V\to \Rbar$ and $g\colon \dual{V}\to \Rbar$  then
\begin{equation*}
  \sup_{k\in \dual{V}}\{\fench{f}(k)-g(k)\} = \sup_{x\in {V}}\{\fench{g}(x)-f(x)\}.
\end{equation*} 
\end{thm}
As mentioned in the introduction, this is stronger than the common assertion that we have a Galois connection as that is equivalent to \[0\ge\dd(f,\fench{g})\ \Longleftrightarrow\ 0\ge\dd(\fench{f},g).\]

On top of the adjointness, we know that  the transform is an $\Rbar$-map, i.e.~a distance non-increasing map so we have \(\dd(f_1,f_2)\ge\dd(\fench{f_1},\fench{f_2})\).  Spelling this out as well gives the following.
\begin{thm}
\label{thm:LegendreIsShort}
Suppose that we have functions $f_1,f_2\colon V\to \Rbar$  then
\begin{equation*}
  \sup_{x\in {V}}\{f_2(x)-f_1(x) \}  \ge  \sup_{k\in \dual{V}}\{\fench{f_1}(k)-\fench{f_2}(k) \}.
\end{equation*} 
\end{thm}
\noindent There is also the analogous result for the reverse transform.
Again, this is stronger than the usual order-theoretic result which just says that \[f_{1}\ge f_{2}\quad\Longrightarrow \quad\fench{f_{1}}\le \fench{f_{2}}.\]

The standard theory of the profunctor nucleus tells us that $f\mapsto \doublefench{f}$ is a closure operator, i.e.~it is idempotent and $f\ge \fench{f}$;
it is widely known (see e.g.~\cite{Rockafellar})  that it is the operation of taking the lower semi-continuous, convex hull of a function, but it is worth sketching a proof of this in the language we are using here.
\begin{thm}
For a function $f\colon V\to \Rbar$, its closure $\doublefench{f}$ is the lower semicontinuous, convex hull of $f$.
\end{thm}
\begin{proof}[Sketch of proof]
First let's recall what the lower semicontinuous, convex hull of a function $f$ is.  On the one hand it is the greatest lower semi-continuous, convex function which is pointwise less than or equal to $f$; another description is as follows.  

A supporting hyperplane of such a function $f$ is a non-vertical affine hyperplane in $V\times \Rbar$ which touches the graph of $f$ but which does not pass over it.  
The lower semicontinuous, convex hull of $f$ can defined to be the function whose graph is the envelope of the supporting hyperplanes of $f$.   

Each non-vertical affine hyperplane is the graph of a function of the form $x\mapsto \langle x,k\rangle+a$ where $k$ is a linear function on $V$ and $a$ is a constant.  The supporting hyperplanes of $f$ correspond to the functions of the form $k-\dd(f,k)$, so we translate the hyperplane vertically until it touches the graph of $f$, and the definition of $\dd(f,k)$ ensures that is the amount we translate by.  The function $k-\dd(f,k)$ is written in the categorical language as $\dd(f,k)\pitchfork k$.  The envelope of the supporting hyperplanes then corresponds to the pointwise supremum of these functions, so in other words, the lower semicontinuous convex hull of $f$ can be written as $\bigproduct_{k\in \dual{V}}  \dd(f,k)\pitchfork k$.  Now observe that $\dd(f,k)=\fench{f}(k)$ so that the given expression is exactly $\fench{(\fench{f})}$ as required.
\end{proof}
It might be of interest to category theorists who skipped the proof that the lower semicontinuous, convex hull of $f$ can be written as a weighted limit: the diagram is the embedding $J\colon \dual{V} \to \pre{V}$ of linear functions on $V$ into the space of all functions on $V$, and the weighting $\dd(f,{-})\colon \dual{V}\to \Rbar$ is the distance-from-$f$ function,
  \[\doublefench{f}=\colim{\dd(f,{-})}{J}=\bigproduct_{k\in \dual{V}}  \dd(f,k)\pitchfork k.\]
Another way of viewing this (see also~\cite[Section~3]{CohenGaubertQuadrat:DualityAndSeparation}) is as projecting $f$ into the subcategory of limits of linear presheaves.

The nucleus here is the isomorphism between the two $\Rbar$-spaces of closed functions.  Writing $\Cvx(V,\Rbar)$ for the $\Rbar$ space of lower semicontinuous, convex functions on $V$, we see that 
 the Legendre-Fenchel transform gives an isomorphism of $\Rbar$-spaces:
 \[\Cvx(V,\Rbar)\cong\Cvx({\dual{V}},\Rbar).\]
The fact that these are isomorphic as sets is the reasonably well-known Legendre-Fenchel duality.  The fact that they are isomorphic as $\Rbar$-spaces is a version of the slightly more obscure Toland-Singer duality~\cite{Toland:Duality,Singer:Duality}, which we spell out here.
\begin{thm}
\label{thm:TolandSingerWeak}
If $f_1,f_2\colon V\to \Rbar$ are lower semi-continuous, convex functions then \(\dd(f_1,f_2)=\dd(\fench{f_1},\fench{f_2})\), in other words,
\begin{equation*}
  \sup_{x\in {V}}\{f_2(x)-f_1(x) \}  =  \sup_{k\in \dual{V}}\{\fench{f_1}(k)-\fench{f_2}(k) \}.
\end{equation*} 
\end{thm}

This is illustrated by Figure~\ref{Figure:TolandSinger}, where we have already observed that 
\[\dd(f_1,f_2)=1=\dd(\fench{f_1},\fench{f_2}),\qquad \dd(f_2,f_1)=3=\dd(\fench{f_2},\fench{f_1}).\]

Actually, the hypotheses of Toland and Singer are weaker than those in the above theorem in that they only require the function $f_2$ to be lower semicontinuous and convex.  However, this can be easily deduced.
\begin{thm}[Toland-Singer Duality]
If $f_1,f_2\colon V\to \Rbar$ are functions with $f_2$ lower semi-continuous and convex  then \(\dd(f_1,f_2)=\dd(\fench{f_1},\fench{f_2})\), in other words,
\begin{equation*}
  \sup_{x\in {V}}\{f_2(x)-f_1(x) \}  =  \sup_{k\in \dual{V}}\{\fench{f_1}(k)-\fench{f_2}(k) \}.
\end{equation*} 
\end{thm}
\begin{proof}
Because $f_{2}$ is convex and lower semicontinuous we have ${f_{2}}=\doublefench{f_{2}}$, so from the $\Rbar$-adjunction (Theorem~\ref{thm:LFadjunction}) we find
 $\dd(\fench{f_1},\fench{f_2})=\dd({f_1},\doublefench{f_2})=\dd({f_1},{f_2})$, as required.
\end{proof}

We can now look at the interpretation of the fact that the nucleus is both complete and cocomplete (Section~\ref{Section:NucleusConstruction}).
  We obtain the formulas for the limits and colimits from Theorem~\ref{thm:InvariantPartComplete} and Section~\ref{Section:PresheavesComplete}.   
 \begin{thm}
 \label{thm:CvxLimitsColimits}
 The set of lower semicontinuous convex functions $\Cvx(V,\Rbar)$  has limits and colimits given as follows.  For $f,f_{i}\in \Cvx(V,\Rbar)$, $a\in \Rbar$, $x\in V$ and $\inf f_{i}$ being the point-wise infimum,
   \begin{align*}
   \bigl(\bigproduct f_i\bigr)(x) &= \sup(f_i(x));  
   &(a\cotensor f)(x) &=f(x)-a;\\
  \bigl(\bigcoproduct f_i\bigr)(x) &= \doublefench{(\inf f_i)}(x); 
  & (a\odot f)(x) &=f(x)+a.
\end{align*}
\end{thm}

A quick sketch with two quadratic functions will show that in the general the pointwise infimum of convex functions is not convex, which justifies the need for taking the lower semicontinuous, convex hull.   For the tensor operation, we should similarly translate and then take the lower semicontinuous, convex hull, but it is clear that convexity is invariant under such a translation, which means that taking the hull is unnecessary.
By Section~\ref{sec:completeness}, the completeness and cocompleteness gives the nucleus two distinct structures as a module over the idempotent semiring $(\Rbar,\min,{+})$ as described in the introduction.

\end{document}

%% file: SWi.tex
\usepackage{amsmath}
\usepackage{amsthm}
\usepackage{amssymb}

\usepackage{rotating}

\usepackage{sectsty}
\chapternumberfont{\sffamily\bfseries\LARGE\sectionrule{3ex}{3pt}{-1ex}{1pt}}
\sectionfont{\rmfamily\bfseries\sectionrule{3ex}{0pt}{-1ex}{1pt}}
\subsectionfont{\rmfamily\bfseries\large}
\subsubsectionfont{\rmfamily\bfseries}

\usepackage{pxfonts}
\usepackage{mathpazo}

\allowdisplaybreaks

\usepackage{etex}

\usepackage{graphicx}
\usepackage{mathrsfs}
\usepackage{color}

\usepackage{ctable}
\usepackage{multirow}
\newcommand{\otoprule}{\midrule[\heavyrulewidth]} 

\newcommand{\define}[1]{\textbf{#1}}
\newcommand{\M}{\mathcal{M}}

\newcommand{\I}{\mathbb I}
\newcommand{\J}{\mathbb J}
\newcommand{\Irel}{\mathbin{\I}}
\newcommand{\R}{\mathbb R}

\newcommand{\CC}{\mathbb C}
\newcommand{\Z}{\mathbb Z}

\newcommand{\C}{\mathcal C}
\newcommand{\D}{\mathcal D}
\newcommand{\E}{\mathcal E}

\newcommand{\V}{\mathcal{V}}

\newcommand{\PP}{\mathcal{P}}

\newcommand{\A}{\mathcal A}
\newcommand{\B}{\mathcal B}

\newcommand{\Rbar}{\overline{\R}}
\newcommand{\Zbar}{\overline{\Z}}
\newcommand{\greeq}{\succcurlyeq}
\renewcommand{\ge}{\geqslant}

\renewcommand{\le}{\leqslant}
\newcommand{\relation}{\mathrel{\lJoin}}
\newcommand{\Rrelation}{\mathrel{\le_R}}
\newcommand{\Srelation}{\mathrel{\le_S}}
\newcommand{\RSrelation}{\mathrel{\le_{RS}}}
\newcommand{\RtoSrelation}{\mathrel{\le_{[RS]}}}

\newcommand{\one}{\mathbb{1}}

\newcommand{\lefttruthval}{\left\lceil}
\newcommand{\righttruthval}{\right\rceil}

\DeclareFontFamily{U}  {MnSymbolA}{}

\DeclareSymbolFont{MnSyA}         {U}  {MnSymbolA}{m}{n}
\DeclareFontShape{U}{MnSymbolA}{m}{n}{
    <-6>  MnSymbolA5
   <6-7>  MnSymbolA6
   <7-8>  MnSymbolA7
   <8-9>  MnSymbolA8
   <9-10> MnSymbolA9
  <10-12> MnSymbolA10
  <12->   MnSymbolA12}{}
\DeclareFontShape{U}{MnSymbolA}{b}{n}{
    <-6>  MnSymbolA-Bold5
   <6-7>  MnSymbolA-Bold6
   <7-8>  MnSymbolA-Bold7
   <8-9>  MnSymbolA-Bold8
   <9-10> MnSymbolA-Bold9
  <10-12> MnSymbolA-Bold10
  <12->   MnSymbolA-Bold12}{}

\newcommand{\colim}[2]{\operatorname{colim}_{#1}(#2)}
\renewcommand{\lim}[2]{\operatorname{lim}_{#1}(#2)}

\DeclareFontFamily{U}{mathx}{}
\DeclareFontShape{U}{mathx}{m}{n}{
  <5> <6> <7> <8> <9> <10>
  <10.95> <12> <14.4> <17.28> <20.74> <24.88>
  mathx10
  }{}
\DeclareSymbolFont{mathx}{U}{mathx}{m}{n}
\DeclareMathAccent{\widecheck}{0}{mathx}{"71}

\newcommand{\copre}[1]{\widecheck{#1}}
\newcommand{\presheaf}[1]{\widehat{#1}}
\newcommand{\opcopre}[1]{\widecheck{#1}}

\newcommand{\pre}[1]{\widehat{#1}}

\newcommand{\underlying}[1]{\underline{#1}}

\usepackage{microtype}
\usepackage{hyperref}

\newcommand{\cl}{\mathrm{cl}}

\DeclareMathOperator{\dd}{d}

\newcommand{\op}{{\mathrm{op}}}

\newcommand{\truefalse}{\mathrm{Truth}}
\newcommand{\Truth}{\truefalse}

\newcommand{\true}{\mathrm{true}}
\newcommand{\false}{\mathrm{false}}

\newcommand{\bigwedgie}{\textstyle\bigwedge}
\newcommand{\biggyvee}{\textstyle\bigvee}

\newcommand{\logand}{\mathbin{\&}}

\newcommand{\entails}{\vdash}
\newcommand{\curlyge}{\succcurlyeq}

\newcommand{\Set}{\mathsf{Set}}

\newcommand{\Vcat}{\V\mathsf{Cat}}

\newcommand{\cotensor}{\pitchfork}

\newcommand{\isomorphic}{\cong}

\newcommand{\equivalent}{\simeq}

\newcommand{\pt}{\star}

\newcommand{\prof}{\rightsquigarrow} 
\newcommand{\powerset}{\PP}
\newcommand{\bigproduct}{\textstyle\prod}
\newcommand{\bigcoproduct}{\textstyle\coprod}

\newcommand{\ctensor}{\odot}

\renewcommand{\phi}{\varphi}
\renewcommand{\epsilon}{\varepsilon}

\newcommand{\scare}[1]{`#1'}

\DeclareMathOperator{\Inv}{Inv}

 \DeclareMathOperator{\Fun}{Fun}

 \DeclareMathOperator{\Image}{Im}

 \DeclareMathOperator{\Nat}{Nat}

\DeclareMathOperator{\Ob}{Ob}

\DeclareMathOperator{\Aut}{Aut}

\DeclareMathOperator{\Yoneda}{\mathcal{Y}}
\DeclareMathOperator{\Fix}{Fix}

  \usepackage{mathtools}
  \DeclareMathOperator{\ob}{ob}
  
  \DeclareMathOperator{\Cvx}{Cvx}

\DeclareMathOperator*{\bigforall}{\forall}

\DeclareMathOperator{\fenchb}{\mathbb{L}^\ast}
\DeclareMathOperator{\rfenchb}{\mathbb{L}_\ast}

\newcommand{\fenchbop}[1]{\fenchb(#1)}
\newcommand{\rfenchbop}[1]{\rfenchb(#1)}

\newcommand{\fench}[1]{#1^{\star}}
\newcommand{\doublefench}[1]{#1^{\star\star}}

\newcommand{\dual}[1]{{#1}^{\#}}

  \newtheorem*{Slogan*}{Slogan}

\newtheorem{thm}{Theorem}

\theoremstyle{definition}

\newcommand{\coleqq}{:=}

%% file: Legendre_xxx.bbl
\begin{thebibliography}{99}

\bibitem{Arnold:MathematicalMethods} 
    V.~I.~Arnold, 
    Mathematical Methods in Classical Mechanics, 
    second edition, Graduate Texts in Mathematics~60, Springer, 1997.
\bibitem{Belohlavek:FuzzyGaloisConnections}
    R.~B\v elohl\'avek,
    \textit{Fuzzy Galois Connections},
    Mathematical Logic Quarterly \textbf{45} (1999), 497--504.
\bibitem{Borceux:Handbook2} F.~Borceux,
  {Handbook of Categorical Algebra 2: Categories and Structures},
  Encyclopedia of Mathematics and its Applications \textbf{51}, 
  Cambridge University Press, 1994.
\bibitem{CohenGaubertQuadrat:DualityAndSeparation}
    G.~Cohen, S.~Gaubert and J.-P.~Quadrat,
    \textit{Duality and separation theorems in idempotent semimodules},
    Linear Algebra and its Applications \textbf{379} (2004) 395--422.
\bibitem{DevelinSturmfels:Tropical} 
    M.~Develin and B.~Sturmfels,
    \textit{Tropical convexity},
    Documenta Mathematica \textbf{9} (2004) 1--27 (electronic).
\bibitem{Fujii:BThesis} 
    S.~Fujii, 
    \textit{A categorical approach to L-convexity}, 
    Bachelor's Thesis, University of Tokyo, February 2014.
\bibitem{GutierrezGarciaEtAl:FuzzyGaloisConnections}
    J.~Guti\'errez Garc\'ia, I.~Mardones-P\'erez, M.~A.~de Prada Vicente and D.~Zhang,
    \textit{Fuzzy Galois connections categorically},
    Mathematical Logic Quarterly, \textbf{56} (2010) 131--147.  
    DOI 10.1002/malq.200810044
\bibitem{Kelly:EnrichedCategoryTheory}
    G.~M.~Kelly,
    \textit{Basic concepts of enriched category theory},
    Reprints in Theory and Applications of Categories \textbf{10} 
    (2005)  pp.~vi+137.
\bibitem{LambekScott}
    J.~Lambek and P.~J.~Scott,
    Introduction to Higher Order Categorical Logic,
    Cambridge Studies in Advanced Mathematics \textbf{7},
    Cambridge University Press, 1986.
\bibitem{Lawvere:MetricSpaces} 
    F.~W.~Lawvere,
    \textit{Metric spaces, generalized logic, and closed categories},
    Reprints in Theory and Applications of Categories \textbf{1} (2002) 1--37.
%
\bibitem{Lawvere:TakingCategoriesSeriously} 
    F.~W.~Lawvere,
    \textit{Taking categories seriously},
    Reprints in Theory and Applications of Categories \textbf{8} (2005) 1--24.
%
\bibitem{Lawvere:Entropy} 
    F.~W.~Lawvere,
    \textit{State categories, closed categories, and the existence of semi-continuous entropy functions},
    IMA Research Report 86, University of Minnesota (1986).
\bibitem{Leinster:Magnitude}
    T.~Leinster,
    \textit{The magnitude of metric spaces},
    Documenta Mathematica \textbf{18} (2013) 857--905.
%
\bibitem{LeinsterCobbold:Diversity} 
    T.~Leinster and C.~A.~Cobbold,
    \textit{Measuring diversity: the importance of species similarity},
    Ecology \textbf{93}  (2012)  477--489. 
\bibitem{MacLane:CWM} 
    S.~Mac~Lane, 
    Categories for the Working Mathematician, 
    second edition, Graduate Texts in Mathematics~5, Springer, 1997.
\bibitem{Murota:DiscreteConvexAnalysis}
    K.~Murota, 
    Discrete Convex Analysis, 
    SIAM, 2003.
\bibitem{Pavlovic:FCA} 
    D.~Pavlovic, 
    \textit{Quantitiative Concept Analysis}, 
    in Formal Concept Analysis, Lecture Notes in Computer Science, Volume 7278 (2012)  260--277;  
    see also \href{http://arxiv.org/abs/1204.5802}{arxiv:1204.5802}
\bibitem{Rockafellar} 
    R.~T.~Rockafellar, 
    Convex Analysis, 
    Princeton University Press, 1972.  See also 
    \url{http://www.math.washington.edu/~rtr/papers.html}
\bibitem{ShenZhang:CategoriesOverAQuantaloid}
    L.~Shen and D.~Zhang,
    \textit{Categories enriched over a quantaloid: Isbell adjunctions 
    and Kan adjunctions},
    Theory and Applications of Categories \textbf{28} (2013) 577--615.
\bibitem{Singer:Duality} I.~Singer,
    \textit{A Fenchel-Rockafellar type duality theorem for maximisation},
    Bulletin of the Australian Mathematical Society \textbf{20} (1979) 193--198.
\bibitem{Toland:Duality} J.~F.~Toland,
    \textit{A duality principle for nonconvex optimization and the calculus of variation},
    Archive for Rational Mechanics and Analysis \textbf{71} (1979) 41--61.
%
\bibitem{Tourchette} 
    H.~Touchette, 
    \textit{Legendre-Fenchel transforms in a nutshell},
    unpublished report, Queen Mary University of London, 2005.
\bibitem{Willerton:Isbell} S.~Willerton,
  \textit{Tight spans, Isbell completions and semi-tropical modules},
  Theory and Applications of Categories  \textbf{28} (2013) 696--732.

\end{thebibliography}
